\documentclass[11pt]{amsart}
\usepackage[numbers, square]{natbib}
\usepackage{amssymb}
\usepackage{amsmath}
\usepackage{amsfonts}
\usepackage{geometry}
\usepackage{amsthm}
\usepackage{hyperref}
\usepackage{wrapfig}
\usepackage[dvipsnames]{xcolor}

\providecommand{\U}[1]{\protect\rule{.1in}{.1in}}
\theoremstyle{plain}

\newtheorem{corollary}{Corollary}

\newtheorem{lemma}{Lemma}

\newtheorem{proposition}{Proposition}

\newtheorem{theorem}{Theorem}
\numberwithin{equation}{section}

\begin{document}

	\title[Quantitative unique continuation]{ Quantitative unique continuation for elliptic equations with H\"older continuous potentials}
	
	\author{Long Teng}
	\address{Department of Mathematics, Louisiana State University, Baton Rouge, LA, USA}
	\email{lteng2@lsu.edu}

	\author{Zhiwei Wang}
    \address{Department of Mathematics, Louisiana State University, Baton Rouge, LA, USA}
    \email{zwang30@lsu.edu}
    %\address{ Department of Mathematics, University of Science and Techonology
%of China, Anhui Hefei, China}
	%\email{wzhiwei@mail.ustc.edu.cn}
	
    \author{Jiuyi Zhu}
	\address{Department of Mathematics, Louisiana State University, Baton Rouge, LA, USA}
	\email{zhu@math.lsu.edu}
	%	\date{\today}	
	
		\begin{abstract}
			We study quantitative unique continuation for second order elliptic equations with lower-order terms of H\"older regularity via a weighted frequency function method.
			We establish quantitative three-ball inequalities and corresponding vanishing-order bounds for  Schr\"odinger equations with H\"older potentials and H\"older gradient terms, and corresponding results for elliptic equations with variable leading coefficients. 
			Our results are quantitative with explicit dependence of  H\"older  norms in the three-ball inequalities. These fill in the gap for quantitative unique continuation between bounded potentials and $C^1$ potentials.
		\end{abstract}

		\maketitle
				
				\section{Introduction}
				\addcontentsline{toc}{section}{Introduction}

              In this paper,  we study quantitative unique continuation results for second-order elliptic equations with H\"older continuous lower-order terms. 
              %Using a weighted frequency function approach combined with mollification of H\"older regularity, we prove quantitative three-ball inequalities and explicit bounds on the order of vanishing for solutions of Schr\"odinger equations with H\"older continuous potentials. 
               % In particular, for solutions of $-\Delta u + V(x)u = 0$ where $V(x) \in C^{0,\beta}(B_{10})$, $||V||_{C^{0, \beta}(B_{10})} \leq M$, we obtain a vanishing order bound of $C(1 + M^{2/{(\beta + 3)}})$. This bound interpolates between the case $C^1$ obtained by Zhu and the bounded potential case due to Kenig, and recover the known optimal exponents at both endpoints. 
Quantitative unique continuation studies the quantitative behavior of strong unique continuation property (SUCP). The strong unique continuation property is a basic property concerning the uniqueness of partial differential equations.
We say the strong unique continuation property  for some partial differential equation holds  if vanishing of infinite order at one point implies the solution  is trivial. The strong unique continuation property holds for a wide range of partial differential equations including second order uniformly elliptic equations. 
If the solution is non-trivial and the strong unique continuation property holds, then the solution can not vanish of infinite order.  We want to ask how fast the solution vanishes and how the vanishing order depends on the potential and gradient terms in the equations. For a smooth function, vanishing order at a given point is the order of the first non-zero derivative such that all lower order derivatives vanish at the point. For non-smooth function, the vanishing order is the first non-zero $N$ such that ${|B_r|^{-1/2}} \|u\|_{L^2(B_r)}=O(r^N)$ as $r\to 0$, where $B_r$ is the Euclidean ball with radius $r$ centered at origin.
To obtain the vanishing order, one usually need to obtain the quantitative {three-ball inequalities} (or Hadamard's three-sphere inequality) or {doubling estimates}.
				%Let $L$ be a second-order elliptic operator.
				%We say that $L$ has the \emph{unique continuation property} (UCP) if any solution of $Lu=0$ which vanishes on a nonempty open set must vanish identically.
				%We say that $L$ has the \emph{strong unique continuation property} (SUCP) if any solution of $Lu=0$ which vanishes to infinite order at a point must be identically zero.
				%These qualitative properties are indispensable in spectral theory and geometric analysis, and they also underlie quantitative statements about nodal sets and related growth phenomena for eigenfunctions; see, for instance, \cite{Lin91,Zelditch08} and the references therein.
				%Beyond qualitative SUCP, one seeks \emph{quantitative} information.
				%A typical quantitative question asks for an explicit upper bound on the \emph{order of vanishing} of a nontrivial solution at a point in terms of norms of the coefficients.
				%Equivalently, one asks for lower bounds of the type
				%\[
				%\|u\|_{L^{2}(B_{r})}\gtrsim r^{\,\mathcal N}
				%\quad\text{as }r\downarrow 0,
				%\]
				%where the exponent $\mathcal N$ is controlled by the coefficient size.
                %which provide scale-invariant ways to measure propagation of smallness.
				A typical three-ball inequality asserts that 
				\[
				\|u\|_{L^{2}(B_{r_{2}})}
				\le
				C\|u\|_{L^{2}(B_{r_{1}})}^{\theta}\,
				\|u\|_{L^{2}(B_{r_{3}})}^{1-\theta}\,
				\]
				for $0<r_{1}<r_{2}<r_{3}$ and  some $\theta\in(0,1)$.
				Iterating such inequalities with a propagation-of-smallness argument  produces an explicit bound of vanishing order.
				
			Let us briefly review some literature on the strong unique continuation property for second order elliptic equations. The  strong unique continuation property relies crucially on the regularity of the potentials and  leading coefficients in the elliptic equations.
				For Schr\"odinger-type equations $\Delta u + V(x) u=0$,
			 Jerison and Kenig \cite{JK85} established strong unique continuation property under the assumption $V(x)\in L_{\mathrm{loc}}^{\frac{n}{2}}$, which is the borderline in the Lebesgue space. That is, the strong unique continuation property fails if $V(x)\in L^p_{\mathrm{loc}}$ for $p<\frac{n}{2}.$ For the general elliptic equations  with gradient terms 
                $ {\rm div}(A(x) \nabla u) +W(x)\cdot\nabla u+ V(x) u=0$,
             Koch and Tataru \cite{KT01}  showed the SUCP holds 
             for $W(x)\in L^{n+\epsilon}_{\mathrm{loc}}$ and $V(x)\in L_{\mathrm{loc}}^{\frac{n}{2}}$ for any small $\epsilon>0$. See e.g.  \cite{S90},\cite{W93}, and references therein for other literature on the study of unique  continuation property.
             Generally speaking, the Lipschitz continuous leading coefficient $A(x)$ is necessary to have the strong unique continuation property, see e.g. \cite{P63}. 
             %whenever the leading coefficients are Lipschitz continuous and the lower order terms satisfy
				%very general conditions.
				%On the other hand,  constructed counterexamples showing that SUCP can fail when the leading coefficients are less regular than Lipschitz.

                For the quantitative unique continuation property, the interesting case arises from the study of Laplace eigenfunctions. For Laplace
eigenfunctions on a compact smooth Riemannian manifold $\mathcal{M}$,
\begin{align}
\label{eigen-1}
    -\triangle_g \phi_\lambda=\lambda \phi_\lambda \ \ \mbox{on} \ \mathcal{M},
\end{align}
Donnelly and Fefferman in [DF88] show the sharp and maximal vanishing order of $\phi_\lambda$ is everywhere
less than $C \sqrt{\lambda}$, here $C$ only depends on the manifold $\mathcal{M}$. This sharpness of $C \sqrt{\lambda}$ can be seen from spherical harmonics.
                %Thus, even before considering quantitative estimates, there is a sharp qualitative threshold in the leading-order regularity.
			For the Schr\"odinger equations
				\begin{align}
				-\Delta u + V(x)u=0 \quad \mbox{in} \ B_{R}
                \label{shrod-1}
				\end{align}
				with non-trivial solutions $u$ and $\|V\|_{L^{\infty}}\le M,$       
Bourgain and Kenig \cite{BK05} considered the quantitative unique continuation for (\ref{shrod-1}) with a background from Anderson localization for the Bernoulli model. Using Carleman estimates, it was shown  in \cite{BK05} that the maximal order of vanishing is bounded by  $CM^{2/3}$.  This dependence is sharp for complex-valued potentials based on Meshkov's counter-example in  \cite{M92}. 
%See also e.g. some studies on 

For the real-valued bounded potential $V(x)$, the study of vanishing order is deeply connected to the Landis conjecture.   Landis  \cite{KL91} in the late 1960s conjectured that the
solution  to $\triangle u- V(x)u=0 $ in $\mathbb R^n$
is trivial if $|u|\leq e^{-C|x|^{1+\epsilon}}$
for any $\epsilon>0$. The Landis’ conjecture was settled recently in \cite{LMNN25}  in
the plane. See also e.g. \cite{KSW15}, \cite{DKW19} for the  study  of this conjecture in $\mathbb R^2$ with additional assumptions on the potentials. For study of the quantitative unique continuation for second order elliptic equations with singular potential $V(x)\in L^p$  for $p>\frac{n}{2}$ and with singular gradient potential $W(x)\in L^p$ for $q>n$, see e.g. \cite{DZ18}, \cite{DZ19}.

 Under some strong regularity assumptions on the potential $V(x)$, the order of vanishing  can be improved as indicated for Laplace eigenfunctions (\ref{eigen-1}). For the Schr\"odinger equations (\ref{shrod-1}) with $\|V\|_{C^1}\leq M$, it was shown in \cite{B12} and \cite{Z16} independently by different methods that the sharp vanishing order is $C \sqrt{M}$, which improves the early work by Kukavica \cite{K98}. 
 
Generally speaking, there are two principal methods to study unique continuation property: {Carleman estimates} and {frequency functions}.
Carleman estimates \cite{C39} are weighted integral inequalities. The Carleman weight function needs to have some convex conditions. The construction of a suitable Carleman weight function is one of major difficulties in establishing and applying Carleman estimates.
The frequency-function approach, originating from Almgren
for harmonic functions,  was adapted by Garofalo--Lin  \cite{GL86} to second order elliptic equations for strong unique continuation. One usually need to obtain some  monotonicity result for 
the frequency function. The variational structure of the equations is essential to build the frequency functions. These two methods do not distinguish the real-valued and complex-valued potentials.
            
            Recently, the frequency function approach was revisited by Davey in \cite{D25} to study the three-ball inequality or vanishing order for (\ref{shrod-1}) with bounded potential $\|V\|_{L^{\infty}}\le M,$ which has been obtained by Carleman estimates in \cite{BK05}. A modified frequency function,   was used to derive the quantitative three-ball inequality. This modified frequency was introduced in \cite{K00}, then studied in \cite{Z16} for Schr\"odinger equations with  $C^1$ potentials. See also \cite{BG16} for its application on the boundary for $C^1$ potentials. In particular, The modified weighted function has been studied for higher order elliptic equations in \cite{Z16}.
            For the $L^\infty$ potentials, an insight of completing  the square, instead of applying divergence theorem to $V(x)\in C^1$ as in \cite{Z16}, was used to derive almost monotonicity in \cite{D25}.

            For H\"older continuous potentials $V(x)$, we regularize the potential by a smooth potential and a bounded potential. Using the weighted frequency function, we obtain the following quantitative three-ball inequality.

				\begin{theorem}
			\label{thm:threeballs-I-opt-correct}
					Let $u$ be the solution of  (\ref{shrod-1}) with $R\geq 1$.
					Assume $\|V\|_{C^{0,\beta}(B_R)}\leq M$ for $0<\beta<1$ and some large constant  $M$.
					Let $0<r_1<r_2<2r_2<r_3< \frac{R}{2}$.
					Then there exists $C=C(n,\beta)>0$ such that
\begin{equation}\label{eq:threeballs-I-final}
	\|u\|_{L^2(B_{r_2})}\le \exp\!\left\{\,C\,M^{\frac{2}{\beta+3}}\,R^{\frac{4+2\beta}{\beta+3}}\,\right\}\|u\|_{L^2(B_{r_1})}^{\theta}\|u\|_{L^2(B_{r_3})} ^{1-\theta}\,			
					\end{equation}
                    for $0<\theta=\frac{\log \frac{r_3}{2r_2}}{\log \frac{r_3}{r_1} }<1$.
				\end{theorem}

    The quantitative three-ball inequality in Theorem \ref{thm:threeballs-I-opt-correct} can be used to derive the vanishing order of solutions. We normalize the solutions in (\ref{shrod-1}) as follows.
    Assume 
	\begin{equation}\label{eq:assump-norm}
				\|u\|_{L^\infty(B_{10})}\le C_0,\qquad \|u\|_{L^2(B_1)}\ge 1.
					\end{equation}
				 \begin{corollary}
				    \label{thm:vanishing-I-correct}
					Assume $u$ satisfies   (\ref{shrod-1}) and (\ref{eq:assump-norm}) in $B_{10}$ with  $\|V\|_{C^{0,\beta}(B_{10})}\leq M$ for $0<\beta<1$.
					Then the vanishing order of $u$ in $B_{1/2}$ is at most $CM^{\frac{2}{\beta+3}}$, where $C$ depends on $n$, $\beta$, and $C_0$.
                    	 \end{corollary}
                    
                  %  there exists $C=C(n,\beta,C_0)>0$ 
                 %   such that for all $0<r\le \frac{1}{10}$,
		%\begin{equation}\label{eq:vanishing-I-correct}
				%		\|u\|_{L^2(B_r)}\ge r^{\,C\left(1+M^{\frac{2}{\beta+3}}\right)} .

We also study the quantitative unique continuation for elliptic equations with drift terms,\begin{equation}\label{eq:mainPDE-1}
    	-\Delta u + W(x)\cdot \nabla u = 0 \qquad \text{in } B_R\subset\mathbb{R}^n,
    \end{equation}
    where $W\in C^{0,\beta_0}(B_R;\mathbb{R}^n)$ is H\"older continuous.
    Usually, it is more challenging to study unique continuation or quantitative unique continuation for the gradient term $W(x)$. See e.g. \cite{W93}, \cite{KT01} for the study of the strong unique continuation property for the gradient term $W(x)$ in critical Lebesgue space $L^q$ for $q>n$. See also \cite{J86} for a counter-example of Carleman estimates used to deal with the strong unique continuation for (\ref{eq:mainPDE-1}).
There is some difficulty in obtaining the sharp growth power for the solutions in (\ref{eq:mainPDE-1}) using the weighted frequency function as the proof of Theorem \ref{thm:threeballs-I-opt-correct}. Instead, we adopt  a novel lifting argument to incorporate the regular parts of gradient terms  into the leading coefficients. We are able to show the following three-ball inequality.

    				 \begin{theorem}
			\label{thm:threeballs-III-final}
					Let $0<\beta_0< 1$ and $u$ solve (\ref{eq:mainPDE-1}) with $R\geq 1$.
					Assume $\|W\|_{C^{0,\beta_0}}\leq K$.
					Then there exist some  constants $C(n, \beta_0)$ such that for any $0<r_1<r_2<2r_2<r_3<\frac{R}{2}$,
		\begin{equation}\label{eq:threeballs-III-final}
					\|u\|_{r_2}
    	\le
    	\exp\!\Bigg\{
    	CK^{\frac{2}{\beta_0+1}}R^{\frac{2}{\beta_0+1}}
    	\Bigl[
    	1+\log\!\Bigl(\frac{r_3}{2 r_2}\Bigr)
    	\Bigr]
    	\Bigg\}
    	\,
    	\|u\|_{r_1}^{\theta}\,
    	\|u\|_{r_3}^{1-\theta},
					\end{equation}
					where $0<\theta=\frac{\log r_3-\log 2r_2}
{\log r_3-\log 2r_2+{Ce^{C(r_3-r_1)}}(\log 2r_2-\log r_1)}<1.$
\end{theorem}

Theorem \ref{thm:threeballs-III-final} leads to the following Corollary on the vanishing order of $u$ in  (\ref{eq:mainPDE-1}).
				 \begin{corollary}
				\label{thm:vanishing-III-correct}
Assume $u$ satisfies  (\ref{eq:mainPDE-1})  and (\ref{eq:assump-norm}) in $B_{10}$ with  $\|W\|_{C^{0,\beta_0}(B_{10})}\leq K$ for $0<\beta_0<1$.
Then the vanishing order of $u$ in $B_{1/2}$ is at most $CK^{\frac{2}{\beta_0+1}}$, where $C$ depends on $n$, $\beta_0$, and $C_0$.
\end{corollary}

					%Assume \eqref{eq:assump-norm}. Fix $\sigma=2$, $r_2=1$, $r_3=10$.
					%Then there exists   $C=C(n,\beta,C_0)>0$ such that
					%for all $0<r\le \frac{1}{10}$,
					%\begin{equation}\label{eq:vanishing-III-correct}
					%	%\|u\|_{L^2(B_r)}\ge r^{\,C\left(1+K^{\frac{2}{\beta_0+1}}\right)} .
					%\end{equation}    
				 
					%\begin{equation}\label{eq:vartheta-final}
					%	\vartheta=
					%	\frac{\log r_3-\log(\sigma r_2)}
					%	{\log r_3-\log(\sigma r_2)+3e^{c\Lambda R+C}\big(\log(\sigma r_2)-\log r_1\big)}\in(0,1).
					%\end{equation}
				%\end{theorem}

    For elliptic equations with variable leading coefficients, we are also able to show the following three-ball inequality with the weighted frequency functions. Let  $u$ satisfy
    \begin{align}\label{variable-1}
					-\sum^{n}_{i,j=1}\partial_j\!\big(a_{ij}(x)\partial_i u\big)+V(x)u=0 \qquad \text{in }  B_R .
					\end{align}
	Assume $A=(a_{ij})^n_{i,j=1}$ is symmetric, uniformly elliptic with constants $\lambda,\Lambda$,
					and Lipschitz continuous with coefficient $L_A\geq 1$,
	\begin{align}
    \label{lips-1}
	\lambda |\xi|^2 \leq \sum^{n}_{i, j=1} a_{ij}(x) \xi_i\xi_j\leq \Lambda |\xi|^2 \quad \mbox{for every } \xi\in \mathbb R^n,\ \  \qquad \mbox{and} \qquad  \|\nabla A\|_{L^\infty(B_R)}\le L_A.
	\end{align}
                    
				\begin{theorem}
			\label{thm:threeballs-II-opt-correct}
					Let $0<\beta< 1$ and  $u$ solve (\ref{variable-1}) under the assumptions (\ref{lips-1}).
					Assume $\|V\|_{C^{0,\beta}(B_R)}\leq M$. Then there exist some universal constants $C_1(n, \lambda, \Lambda, L_A, \beta)$ and $C(n, \lambda, \Lambda, \beta)$ such that for any $0<r_1<r_2<2r_2<r_3<\frac{R}{2}$,
                   % define
		%\begin{equation}\label{eq:beta-gamma-part2}
					%	\xi:=e^{-c_0L_A r_1}\log\!\Big(\frac{2r_2}{r_1}\Big),
					%	\qquad
					%	\eta:=e^{-c_0L_A r_3}\log\!\Big(\frac{r_3}{2r_2}\Big),
						%\qquad
						%\theta:=\frac{\eta}{\xi+\eta}\in(0,1),
					%\end{equation}
					%where $c_0>0$ is an absolute constant.
					%Then there exists $C=C(n,\lambda,\Lambda,\beta)>0$ such that
		\begin{equation}\label{eq:threeballs-II-final}
				\|u\|_{L^2(B_{r_2})}\le \exp\!\left\{\,C_1 M^{\frac{2}{\beta+3}} R^{\frac{4(\beta+1)}{\beta+3}} (\log\frac{r_3}{2r_2}+1)\,\right\}\|u\|_{L^2(B_{r_1})}^{\theta}\|u\|_{L^2(B_{r_3})} ^{1-\theta}\,			
					\end{equation}
                    for some $0<\theta=\frac{\log r_3-\log 2r_2}
{\log r_3-\log 2r_2+{e^{CL_A(r_3-r_1)}}(\log 2r_2-\log r_1)}
                    <1$.
                    	\end{theorem}
					
                    %where the optimized exponent is the explicit function
					%\begin{align}\label{eq:Psi2-final}
					%	\Psi_2(M,L_A;R)
						%:={}&
					%	L_A R
					%	+ M^{\frac{2}{\beta+3}}\,R^{\frac{4+2\beta}{\beta+3}}[log\frac{r_3}{2r_2}+1]
					%	+ M^{\frac{\beta+1}{\beta+3}}\,R^{\frac{2}{\beta+3}}
					%	+ L_A\,M^{\frac{\beta+1}{\beta+3}}\,R^{\frac{\beta+5}{\beta+3}}
					%	+ L_A\,M^{\frac{2}{\beta+3}}\,R^{\frac{3\beta^2+10\beta+7}{(\beta+1)(\beta+3)}} .
				%	\end{align}

				%\begin{theorem}[Order of vanishing from Theorem~\ref{thm:threeballs-II-opt-correct}]
				%	\label{thm:vanishing-II-correct}
				%	Assume $u$ solves $-\partial_j(a_{ij}\partial_i u)+Vu=0$ in $B_{10}$ and the hypotheses of
				%	Theorem~\ref{thm:threeballs-II-opt-correct} hold in $B_{10}$ with $R=10$.
				%	Assume \eqref{eq:assump-norm}. Then there exist constants
					%$C=C(n,\lambda,\Lambda,\beta_0)>0$ and $r_\ast=r_\ast(n,\lambda,\Lambda,L_A)>0$ such that
					%for all $0<r\le r_\ast$,
					%\begin{equation}\label{eq:vanishing-II-correct}
					%	\|u\|_{L^2(B_r)}\ge r^{\,C\left(1+\Psi_2(M,L_A;10)\right)} ,
				%	\end{equation}
				%	where $\Psi_2$ is defined in \eqref{eq:Psi2-final}.
				%\end{theorem}

				It is obvious that the exponent $M^{2/(\beta+3)}$ for H\"older potentials  
                in Theorem \ref{thm:vanishing-I-correct} and Corollary \ref{thm:vanishing-I-correct} interpolates between $C^1$ potential and bounded potentials. That is, it matches the $C^1$ potentials with $M^{1/2}$  in 
        \cite{Z16,B12} when $\beta=1$,
				and the bounded potential  with $M^{2/3}$  in
                \cite{BK05} when $\beta=0$.
				To achieve this interpolation results, we adapt the mollification for H\"older continuous functions and the weighted frequency functions, which contribute our first novel technique.
				Our second novel idea concerns the proof of Theorem \ref{thm:threeballs-III-final} with {H\"older gradient terms}.
			    We mention that  there is some difficulty in 
                using weight frequency function as the proof of Theorem  \ref{thm:threeballs-I-opt-correct} to derive the
                sharp power of quantitative unique continuation. After regularizing the H\"older gradient term into regular parts and bounded parts, it seems be impossible to balance the growth power from these two parts to achieve the optimal exponent in the three-ball inequalities. 
             We
                introduce the {lifting} procedure to incorporate the regular parts of the gradient term into the leading coefficients of a new elliptic equations and thus leave only the bounded parts to optimize.  The exponent $K^{\frac{2}{\beta_0+1}}$  for H\"older norms in Theorem \ref{thm:vanishing-III-correct}  interpolates between $C^1$ gradient terms and bounded potentials. That is, it matches the
                 $C^1$ gradient term with $K$  in  \cite{BC14}, \cite{Z15} when $\beta_0=1$, and the bounded gradient term
              with $K^2$ in \cite{D14} when  $\beta_0=0$.
				
				The organization of the paper is as follows.
				Section \ref{basic equation with potential Vu} is devoted to the proof of Theorem \ref{thm:threeballs-I-opt-correct} using the weighted frequency-function  for classical Schr\"odinger equations with H\"older continuous $V(x)$. The goal is to present the complete idea and self-contained proof for this simple model without spending efforts dealing with calculation technicalities. Section~\ref{nabla term and lifting} introduces the lifting method for H\"older gradient terms and prove Theorem \ref{thm:threeballs-III-final} and Corollary \ref{thm:vanishing-III-correct}.
				Section~\ref{general eqution} extends the arguments to second order elliptic equations  with Lipschitz leading coefficients  and prove Theorem \ref{thm:threeballs-II-opt-correct}. Compared with section \ref{basic equation with potential Vu}, there are technicalities to deal with the variable leading coefficients. The Appendix in section 5 include some basic lemma in the proof of Theorem \ref{thm:threeballs-II-opt-correct}.
                The letters $C$, $C_i$, ${c}$ and $c_i$ denote positive constants that do not depend on $V$, $W$ or $u$, and may vary from line to line. The potential $V(x)$ and the gradient term $ W(x)$ can be complex-valued functions. The letters $K, M\geq 1$ are assumed to be some large constants.
				
				%Finally, in Section~\ref{vanishing order} we prove vanishing order, namely, Theorem \ref{thm:vanishing-I-correct}-\ref{thm:vanishing-III-correct}.

\section{Three-ball inequality for classical Schr\"odinger equations} \label{basic equation with potential Vu}

In this section, we study the quantitative unique continuation for the following classical Schr\"odinger equations
	\begin{equation}\label{eq:main}
		-\Delta u + V(x)u = 0 \qquad \text{in } B_R,
	\end{equation}
	where $V \in C^{0,\beta}(B_R)$ and  $0<\beta<1$. Note that $B_R$ is a Euclidean ball centered at origin with radius $R$ and  $R$ is a fixed constant. We skip the dependence of the center for this notation. We aim to provide a self-contained  and detailed proof for Theorem \ref{thm:threeballs-I-opt-correct}. We regularize the H\"older continuous potential with the standard mollifier.
	Let $V_\varepsilon(x) = \rho_\varepsilon * V=\int_{\mathbb R^n}\rho_\varepsilon(x-y)V(y)\,dy$, where
	\[
	\rho_\varepsilon(x)=\frac{1}{\varepsilon^n}\rho\Bigl(\frac{x}{\varepsilon}\Bigr),
	\qquad
	\rho(x)\in C_0^\infty(B_1),
	\qquad
	\int_{B_1}\rho(x) dx=1.
	\]
	Assume
	\begin{align}
	\|V\|_{L^\infty(B_R)}\le M,
	\qquad
	[V]_{C^{0,\beta}(B_R)}\le M_0.
    \label{assump-1}
	\end{align}
	Later on we will choose $\max\{M, M_0\}=M$ since we consider $\|V\|_{C^{0,\beta}(B_R) }\leq M$. We want to see how the H\"older semi-norm $[V]_{C^{0,\beta}(B_R)}$ plays the role in the quantitative estimates.
	We can show the following regularization estimates for $V_\varepsilon$.
	\begin{lemma} \label{lem:moll}
    It holds that
 $\|V_\varepsilon\|_{L^\infty(B_{R-\varepsilon})}\le M$ and
	\begin{equation}\label{eq:moll-est}
		\|\nabla V_\varepsilon\|_{L^\infty(B_{R-\varepsilon})}\le C M\varepsilon^{\beta-1},
		\qquad
		\|V-V_\varepsilon\|_{L^\infty(B_{R-\varepsilon})}\le C M_0 \varepsilon^\beta
	\end{equation}
     for a constant $C=C(n)$.
	\end{lemma}

	\begin{proof}
   
	%	For $x\in B_{R-\varepsilon}$,
		%\[
		%V_\varepsilon(x)=\int_{B_R}\rho_\varepsilon(x-y)V(y)\,dy.
		%\]
	%	Moreover, $|V_\varepsilon(x)|\le M$ and
	%	\[
	%	|\nabla V_\varepsilon(x)|
	%	=\left|\int_{B_R}\nabla\rho_\varepsilon(x-y)V(y)\,dy\right|
	%	\le C K \varepsilon^{\beta-1},
		%\qquad
	%	|V_\varepsilon(x)-V(x)|\le C K \varepsilon^\beta.
		%\]

		For $x\in B_{R-\varepsilon}$, the convolution $V_\varepsilon(x)$ is supported in $B_R$.
    %i.e.$V_\varepsilon(x)=\int_{B_R}\rho_\varepsilon(x-y)V(y)\,dy$. 
		Hence we have 
        \begin{align}
		|V_\varepsilon(x)|
		\le \|V\|_{L^\infty(B_R)}\int_{\mathbb R^n}\rho_\varepsilon(y)\,dy
		\le M.
        \end{align}
Since $V_\varepsilon(x)$ is smooth in $B_{R-\varepsilon}$, it is true that
		\[
		\nabla V_\varepsilon(x)
		=\int_{B_R}\nabla\rho_\varepsilon(x-y)V(y)\,dy
		=\frac{1}{\varepsilon^{n+1}}\int_{\mathbb R^n}\nabla\rho\Bigl(\frac{x-y}{\varepsilon}\Bigr)V(y)\,dy.
		\]
		Set $z=\frac{x-y}{\varepsilon}$. We get
		\[
		\nabla V_\varepsilon(x)
		=\int_{B_1}\varepsilon^{-1}\nabla\rho(z)\,V(x-\varepsilon z)\,dz.
		\]
By the fact that $\int_{B_1}\nabla\rho(z)\,dz=0$, it follows that
		\[
		\nabla V_\varepsilon(x)
		=\int_{B_1}\varepsilon^{-1}\nabla\rho(z)\bigl(V(x-\varepsilon z)-V(x)\bigr)\,dz.
		\]
The H\"older continuity of $V$ yields that
		\[
		|\nabla V_\varepsilon(x)|
		\le \varepsilon^{-1} M_0\int_{B_1}|\nabla\rho(z)|\,|\varepsilon z|^\beta\,dz
		\le C M_0 \varepsilon^{\beta-1}.
		\]
		Finally, we derive
		\[
		|V_\varepsilon(x)-V(x)|
		=\left|\int_{B_R}\bigl(V(y)-V(x)\bigr)\rho_\varepsilon(x-y)\,dy\right|
		\le M_0 \int_{B_R}|x-y|^\beta\rho_\varepsilon(x-y)\,dy
		\le C M_0 \varepsilon^\beta.
		\]
        This completes the proof of the Lemma.
	\end{proof}

    Now we rewrite \eqref{eq:main} as
	\begin{equation}\label{eq:moll-eq}
		-\Delta u + V_\varepsilon u = f_\varepsilon u,
		\qquad
		f_\varepsilon:=V_\varepsilon - V,
		\qquad
		\|f_\varepsilon\|_{L^\infty}\le C M_0 \varepsilon^\beta.
	\end{equation}

Our frequency functions have a weight component. 
    Define the weight function as $e_r(x)=r^2-|x|^2.$
We introduce some notations involving the weight function.
	\begin{align}
		H(r) &:= \int_{B_r} u^2\, e_r^{\alpha-1}\,dx, \quad \quad
		I_0(r) := 2\alpha\int_{B_r} u(x\cdot\nabla u)\, e_r^{\alpha-1}\,dx, \label{notations-1}\\
		D(r) &:= \int_{B_r} |\nabla u|^2\, e_r^\alpha\,dx, \quad \
		L(r) := \int_{B_r} V_\varepsilon u^2\, e_r^\alpha\,dx,\label{notations-2}\\
		E(r) &:= \int_{B_r} f_\varepsilon u^2\, e_r^\alpha\,dx, \quad \quad
		I(r) := D(r)+L(r). \label{notations-3}
	\end{align}
		We define our frequency function as 
        \begin{align}
        \label{frequency-2}
            \mathcal{N}(r):=\frac{I(r)}{H(r)}.
        \end{align}
       %This modifier frequency functions with weights $ e_r^{\alpha}$  that replace boundary integrals by weighted bulk functionals. 
       These modified frequency functions offer more flexibility and can cancel the boundary terms when applying integration by parts arguments. In the following proof, we combine the ideas in \cite{Z16} and \cite{D25}.
We first show a general differentiation formula with derivative with respect to $r$ involving the weight function. It is derived by a simple integration by parts argument. We include the proof to show the conveniences of using this weight function and for complete of the presentation.
	\begin{lemma}\label{lem:diffF}
		For any $G\in C^1$, it holds that
		\begin{equation}\label{eq:diff-formula}
			\frac{d}{dr}\int_{B_r}G\,e_r^\alpha\,dx
			=\frac{2\alpha+n}{r}\int_{B_r}G\,e_r^\alpha\,dx
			+\frac{1}{r}\int_{B_r}x\cdot\nabla G\,e_r^\alpha\,dx.
		\end{equation}
	\end{lemma}
	
	\begin{proof}
		Since $e_r^\alpha$ vanishes on $\partial B_r$, the term with integration on the boundary will disappear  when taking derivatives with respect to $r$. Thus, we have
		\begin{align}
		\frac{d}{dr}\int_{B_r}G\,e_r^\alpha
		=\int_{B_r}G\,\frac{d}{dr}(e_r^\alpha)
		=2\alpha \int_{B_r}G r\,e_r^{\alpha-1}.
        \label{general-1}
		\end{align}
	By the fact that $r^2=e_r+|x|^2$, we have
		\begin{align}
		2\alpha\int_{B_r} r G e_r^{\alpha-1}
		=\frac{2\alpha}{r}\int_{B_r} Ge_r^\alpha
		+\frac{2\alpha}{r}\int_{B_r}G|x|^2 e_r^{\alpha-1}.
        \label{general-2}
		\end{align}
	Integrating by parts also shows that
		\[
		\frac{2\alpha}{r}\int_{B_r}G|x|^2 e_r^{\alpha-1}
		=-\frac{1}{r}\int_{B_r}G\,x\cdot\nabla e_r^\alpha
		=-\frac{1}{r}\int_{B_r}x\cdot\nabla(Ge_r^\alpha)
		+\frac{1}{r}\int_{B_r}x\cdot\nabla G\,e_r^\alpha.
		\]
		Applying the integration by parts argument again, we have
		\begin{align}
		\frac{2\alpha}{r}\int_{B_r}G|x|^2 e_r^{\alpha-1}
		=\frac{n}{r}\int_{B_r}G e_r^\alpha+\frac{1}{r}\int_{B_r}x\cdot\nabla G\,e_r^\alpha.
        \label{general-3}
		\end{align}
        The combination of (\ref{general-1})-(\ref{general-3}) yields \eqref{eq:diff-formula}.
	\end{proof}

         Before we prove the almost monotonicity of frequency function  $\mathcal{N}(r)$,
 we first show some identities.
	\begin{lemma}\label{lem:I0 Simple}
		It holds that
		\begin{equation}\label{eq:I0=I-E}
			I_0(r)=D(r)+L(r)-E(r)=I(r)-E(r).
		\end{equation}
        \begin{equation}\label{eq:Hprime}
			H'(r)=\frac{2\alpha+n-2}{r}H(r)+\frac{1}{\alpha r}I_0(r)
			=\frac{2\alpha+n-2}{r}H(r)+\frac{1}{\alpha r}\bigl(I(r)-E(r)\bigr).
		\end{equation}
	\end{lemma}
	
	\begin{proof}
		Since $\nabla e_r^\alpha=\alpha e_r^{\alpha-1}\nabla e_r$ and $\nabla e_r=-2x$, then
		\begin{align}\label{more-ef1}
		-\int_{B_r}u\nabla u\cdot\nabla e_r^\alpha\,dx
		=2\alpha\int_{B_r}u(x\cdot\nabla u)e_r^{\alpha-1}\,dx
		=I_0(r).
		\end{align}
		Performing integrating by parts argument gives that
		\begin{align}\label{more-ef2}
		-\int_{B_r}u\nabla u\cdot\nabla e_r^\alpha\,dx
		=\int_{B_r}|\nabla u|^2 e_r^\alpha\,dx+\int_{B_r}u\,\Delta u\, e_r^\alpha\,dx.
		\end{align}
		It follows from \eqref{eq:moll-eq} and the last two identities  that
		\[
		\int_{B_r}u\,\Delta u\, e_r^\alpha
		=\int_{B_r}V_\varepsilon u^2 e_r^\alpha-\int_{B_r}f_\varepsilon u^2 e_r^\alpha
		=L(r)-E(r).
		\]
		Thanks to (\ref{more-ef1}) and (\ref{more-ef2}), we obtain $I_0(r)=D(r)+L(r)-E(r)$. This finishes the proof of (\ref{eq:I0=I-E}).

        Recall that $H$ and $I_0$ are introduced in (\ref{notations-1}). Let $G=u^2$ with $\alpha$
	replaced by $\alpha-1$ in Lemma \ref{lem:diffF}. We obtain
     \begin{align*}
    H'(r)&=\frac{2(\alpha-1)+n}{r}\int_{B_r}u^2 e_r^{\alpha-1}\,dx+\frac{2}{r}\int_{B_r}u(x\cdot\nabla u)e_r^{\alpha-1}\,dx    \nonumber\\
    &=\frac{2\alpha+n-2}{r} H(r)+\frac{1}{\alpha r}I_0(r).
    \end{align*}
Together with (\ref{eq:I0=I-E}), we arrive at the proof of (\ref{eq:Hprime}).
\end{proof}

	%	Next we show the identity (\ref{eq:Hprime}).
      %  Taking derivative with respect to $r$ gives that
       
%	Since $r^2=e_r+|x|^2$, we rewrite the last inequality as 
		%\begin{align}
		%	H'(r)
	%		&=2(\alpha-1)\int_{B_r}u^2 \frac{(e_r+|x|^2)}{r}e_r^{\alpha-2}\,dx\nonumber\\
	%		&=\frac{2(\alpha-1)}{r}\int_{B_r}u^2 e_r^{\alpha-1}\,dx
	%		+\frac{2(\alpha-1)}{r}\int_{B_r}u^2|x|^2 e_r^{\alpha-2}\,dx.
    %        \label{right-1}
	%	\end{align}
	%Note that
	%	\[
	%	\nabla e_r^{\alpha-1}=(\alpha-1)e_r^{\alpha-2}\nabla e_r=-2(\alpha-1)x e_r^{\alpha-2}.
	%	\]
	%	For the second term in the right hand side of (\ref{right-1}), we have
	%	\begin{align*}
	%		\frac{2(\alpha-1)}{r}\int_{B_r}u^2|x|^2 e_r^{\alpha-2}\,dx
	%		&=-\frac{1}{r}\int_{B_r}u^2 x\cdot\nabla e_r^{\alpha-1}\,dx\\
	%		&=\frac{2}{r}\int_{B_r}u(x\cdot\nabla u)e_r^{\alpha-1}\,dx
	%		+\frac{n}{r}\int_{B_r}u^2 e_r^{\alpha-1}\,dx\\
	%		&=\frac{1}{\alpha r}I_0(r)+\frac{n}{r}H(r).
	%	\end{align*}
     %   Together with (\ref{right-1} and the last inequality, the identity \eqref{eq:Hprime} is derived.

Now we are able to prove the monotonicity results for the weighted frequency functions.
\begin{proposition}\label{pro-1}
Let $u$ be the solution of (\ref{eq:main}) with $V\in C^{0,\beta}$ satisfying (\ref{assump-1}). Then 
\begin{align}
    \tilde{\mathcal{N}}(r):=\mathcal{N}(r)+P(r)
\end{align}
is non-decreasing with respect to $r$ in $B_{R-\varepsilon}$, where $P(r)=Mr^2+CM_0\varepsilon^{\beta-1}r^3+\frac{CM_0^2\varepsilon^{2\beta}}{\alpha}r^4.$
\end{proposition}
\begin{proof}
	Apply \eqref{eq:diff-formula} with $G=|\nabla u|^2$ to have
	\begin{equation}\label{eq:Dprime-start}
		D'(r)=\frac{2\alpha+n}{r}D(r)+\frac{1}{r}\int_{B_r}x\cdot\nabla(|\nabla u|^2)\,e_r^\alpha\,dx.
	\end{equation}
	Direct calculation  yields that
	\begin{align}\label{eq:xgrad-grad2}
		\frac{1}{r}\int_{B_r}x\cdot\nabla(|\nabla u|^2)\,e_r^\alpha
		&=\frac{2}{r} \int_{B_r} \sum_{i,j}u_{ij}u_i x_j\,e_r^\alpha \nonumber \\
		&=-\frac{2}{r}\int_{B_r}x\cdot\nabla u\,\Delta u\,e_r^\alpha
		-\frac{2}{r}\int_{B_r}|\nabla u|^2 e_r^\alpha
		+\frac{4\alpha}{r}\int_{B_r}(x\cdot\nabla u)^2 e_r^{\alpha-1}.
	\end{align}
It follows from  \eqref{eq:moll-eq}, \eqref{eq:Dprime-start}--\eqref{eq:xgrad-grad2} that
	\begin{equation}\label{eq:Dprime}
		\begin{aligned}
			D'(r)
			&=\frac{2\alpha+n-2}{r}D(r)
			+\frac{4\alpha}{r}\int_{B_r}(x\cdot\nabla u)^2 e_r^{\alpha-1}\\
			&\quad
			-\frac{2}{r}\int_{B_r}x\cdot\nabla u\,V_\varepsilon u\,e_r^\alpha
			+\frac{2}{r}\int_{B_r}x\cdot\nabla u\,f_\varepsilon u\,e_r^\alpha .
		\end{aligned}
	\end{equation}
	
	Similarly, applying \eqref{eq:diff-formula} to $G=V_\varepsilon u^2$ yields that
	\begin{equation}\label{eq:Lprime}
		L'(r)=\frac{2\alpha+n}{r}L(r)
		+\frac{1}{r}\int_{B_r}x\cdot\nabla V_\varepsilon\,u^2\,e_r^\alpha
		+\frac{2}{r}\int_{B_r}V_\varepsilon u\,x\cdot\nabla u\,e_r^\alpha .
	\end{equation}
	
We combine \eqref{eq:Dprime} and \eqref{eq:Lprime} to get
	\begin{equation}\label{eq:Iprime}
		\begin{aligned}
			I'(r)=D'(r)+L'(r)
			&=\frac{2\alpha+n-2}{r}\bigl(D(r)+L(r)\bigr)+\frac{2}{r}L(r)
			+\frac{4\alpha}{r}\int_{B_r}(x\cdot\nabla u)^2 e_r^{\alpha-1}\\
			&\quad
			+\frac{2}{r}\int_{B_r}x\cdot\nabla u\,f_\varepsilon u\,e_r^\alpha
			+\frac{1}{r}\int_{B_r}x\cdot\nabla V_\varepsilon\,u^2\,e_r^\alpha .
		\end{aligned}
	\end{equation}
	
Next we study the derivative of $\mathcal{N}$, that is, 
\begin{align}
    \mathcal{N}'=\frac{I'H-H'I }{H^2}.
\end{align}
From \eqref{eq:Hprime} and (\ref{eq:Iprime}), we get
	\begin{equation}\label{eq:H2Nprime-start}
		\begin{aligned}
			H^2 \mathcal{N}'
			&=H\Bigl(\frac{2\alpha+n-2}{r}(D+L)+\frac{2}{r}L+\frac{4\alpha}{r}\!\int_{B_r}(x\!\cdot\!\nabla u)^2 e_r^{\alpha-1}
			+\frac{2}{r}\!\int_{B_r}x\!\cdot\!\nabla u\,f_\varepsilon u\,e_r^\alpha
			\\
			&\quad+\frac{1}{r}\!\int_{B_r}x\!\cdot\!\nabla V_\varepsilon\,u^2 e_r^\alpha\Bigr)
			-\frac{2\alpha+n-2}{r}HI-\frac{1}{\alpha r}(I-E)I\\
			&=\frac{2}{r}HL
			+H\Bigl(\frac{4\alpha}{r}\!\int_{B_r}(x\!\cdot\!\nabla u)^2 e_r^{\alpha-1}
			+\frac{2}{r}\!\int_{B_r}x\!\cdot\!\nabla u\,f_\varepsilon u\,e_r^\alpha
			+\frac{1}{r}\!\int_{B_r}x\!\cdot\!\nabla V_\varepsilon\,u^2 e_r^\alpha\Bigr) \\
			&\quad -\frac{1}{\alpha r}(I-E)I.
		\end{aligned}
	\end{equation}
	We introduce a function
	\[
	\Phi:=x\cdot\nabla u+\frac{e_r}{4\alpha}f_\varepsilon u.
	\]
	%Then 
	%\[
	%\int_{B_r}\Phi^2 e_r^{\alpha-1}
	%=\int_{B_r}(x\cdot\nabla u)^2 e_r^{\alpha-1}
	%+\int_{B_r}\frac{1}{2\alpha}(x\cdot\nabla u)\,f_\varepsilon u\,e_r^\alpha
	%+\int_{B_r}\frac{1}{16\alpha^2}f_\varepsilon^2u^2 e_r^{\alpha+1}.
	%\]
	Then it holds that
	\begin{equation}\label{eq:cross-term}
		\int_{B_r}(x\cdot\nabla u)\,f_\varepsilon u\,e_r^\alpha
		=
		2\alpha\int_{B_r}\Phi^2 e_r^{\alpha-1}
		-2\alpha\int_{B_r}(x\cdot\nabla u)^2 e_r^{\alpha-1}
		-\frac{1}{8\alpha}\int_{B_r}f_\varepsilon^2u^2 e_r^{\alpha+1}.
	\end{equation}
Plugging \eqref{eq:cross-term} into \eqref{eq:H2Nprime-start} gives that
	\begin{equation}\label{eq:H2Nprime-mid}
		\begin{aligned}
			H^2\mathcal{N}'
			&=\frac{2HL}{r}+\frac{4\alpha H}{r}\int_{B_r}\Phi^2 e_r^{\alpha-1}
			-\frac{H}{4\alpha r}\int_{B_r}f_\varepsilon^2u^2 e_r^{\alpha+1}
			+\frac{H}{r}\int_{B_r}x\cdot\nabla V_\varepsilon\,u^2 e_r^\alpha \\ 
			&\quad -\frac{1}{\alpha r}(I-E)I.
		\end{aligned}
	\end{equation}
Next we introduce
	\[
	J(r):=\int_{B_r}u\Phi\,e_r^{\alpha-1}
	=\int_{B_r}u(x\cdot\nabla u)e_r^{\alpha-1}
	+\frac{1}{4\alpha}\int_{B_r}f_\varepsilon u^2 e_r^\alpha.
	\]
	From Lemma \ref{lem:I0 Simple},   (\ref{notations-1}) and (\ref{notations-3}), we have
	\[
	J(r)=\frac{I-E}{2\alpha}+\frac{E}{4\alpha}=\frac{2I-E}{4\alpha}.
	\]
	Then we also have
	\[
	(I-E)I=\frac{(2I-E)^2}{4}-\frac{E^2}{4}
	=16\alpha^2J^2-\frac{E^2}{4}.
	\]
That is,
	\begin{equation}\label{eq:IEI}
		-\frac{1}{\alpha r}(I-E)I=-\frac{4\alpha}{r}J^2+\frac{E^2}{4\alpha r}.
	\end{equation}
Substituting \eqref{eq:IEI} in \eqref{eq:H2Nprime-mid} gives that
	\begin{equation}\label{eq:H2Nprime-square}
		\begin{aligned}
			H^2 \mathcal{N}'
			&=\frac{2}{r}HL
			+\frac{4\alpha}{r}\Bigl(H\int_{B_r}\Phi^2 e_r^{\alpha-1}-J^2\Bigr)
			+\frac{1}{4\alpha r}\Bigl(E^2-H\int_{B_r}f_\varepsilon^2u^2 e_r^{\alpha+1}\Bigr)\\
			&+\frac{H}{r}\int_{B_r}x\cdot\nabla V_\varepsilon\,u^2 e_r^\alpha .
		\end{aligned}
	\end{equation}
The application of  Cauchy-Schwarz inequality yields that
	\[
	H\int_{B_r}\Phi^2 e_r^{\alpha-1}-J^2
	=
	  \int_{B_r}u^2 e_r^{\alpha-1} \int_{B_r}\Phi^2 e_r^{\alpha-1}-\Bigl(\int_{B_r}u\Phi\,e_r^{\alpha-1}\Bigr)^2
	\ge 0.
	\]
With aid of the last inequality and \eqref{eq:H2Nprime-square}, we derive that
	\begin{equation}\label{eq:H2Nprime-lower}
		H^2 \mathcal{N}'
		\ge \frac{2HL}{r}
		-\frac{H}{4\alpha r}\int_{B_r}f_\varepsilon^2u^2 e_r^{\alpha+1}
		+\frac{H}{r}\int_{B_r}x\cdot\nabla V_\varepsilon\,u^2 e_r^\alpha .
	\end{equation}
	
	Thanks to Lemma \ref{lem:moll}, we have $\|f_\varepsilon\|_{L^\infty}^2\le C M_0^2\varepsilon^{2\beta}$ and
\begin{align}
	|L(r)|
	\le \|V_\varepsilon\|_{L^\infty}\int_{B_r}u^2 e_r^\alpha
	\le  Mr^2 \,H(r).
    \label{mono-11}
	\end{align}
	Note that $e_r^{\alpha+1}\le r^4 e_r^{\alpha-1}$ and $e_r^\alpha\le r^2 e_r^{\alpha-1}$. Thus, we obtain
	\begin{align}
	\left|\frac{H}{4\alpha r}\int_{B_r}f_\varepsilon^2u^2 e_r^{\alpha+1}\right|
	\le \frac{C M_0^2\varepsilon^{2\beta}r^4}{4\alpha r}H^2,
	\qquad
	\left|\frac{H}{r}\int_{B_r}x\cdot\nabla V_\varepsilon\,u^2 e_r^\alpha\right|
	\le C M_0\varepsilon^{\beta-1}r^2 H^2.
    \label{mono-12}
	\end{align}
	Therefore, the combination of \eqref{eq:H2Nprime-lower}, (\ref{mono-11}), and (\ref{mono-12})  implies
	\begin{equation}\label{eq:Nprime}
        \mathcal{N}'(r)\ge -2Mr - C M_0r^2\varepsilon^{\beta-1}-\frac{C M_0^2 r^3\varepsilon^{2\beta}}{4\alpha}.
		\qquad
	\end{equation}
   Set $P(r):=Mr^2+CM_0\varepsilon^{\beta-1}r^3+\frac{CM_0^2\varepsilon^{2\beta}}{\alpha}r^4.$
    	%Let
	%\[
	%a_1=2M,
	%\qquad
	%a_2=C K\varepsilon^{\beta-1},
	%\qquad
	%a_3=\frac{C K^2\varepsilon^{2\beta}}{4\alpha},
	%\]
	%and define
	%\[
	%P(r)=\frac{a_1}{2}r^2+\frac{a_2}{3}r^3+\frac{a_3}{4}r^4,
	%\qquad
	%\widetilde N(r):=N(r)+P(r).
	%\]
	Hence \eqref{eq:Nprime} verifies the conclusion of the proposition.
    \end{proof}

    Thanks to the monotonicity results of frequency functions in the Proposition \ref{pro-1}, we derive the three-ball inequality in Theorem \ref{thm:threeballs-I-opt-correct}.
\begin{proof}[Proof of Theorem \ref{thm:threeballs-I-opt-correct}]
	It follows from \eqref{eq:Hprime} and definition of the $\mathcal{N}$ that
	\begin{align*}
	\frac{H'}{H}
	&=\frac{2\alpha+n-2}{r}+\frac{I-E}{\alpha r H} \nonumber \\
	&=\frac{2\alpha+n-2}{r}+\frac{1}{\alpha r}\bigl[\mathcal{N}+P\bigr]-\frac{P}{\alpha r}-\frac{E}{\alpha r H}.
	\end{align*}
From Lemma \ref{lem:moll} and (\ref{notations-3}), we know
	\[
	\left|\frac{E}{\alpha r H}\right|
	\le \frac{1}{\alpha r H}\int_{B_r}|f_\varepsilon|u^2 e_r^\alpha
	\le \frac{C M_0\varepsilon^\beta r^2 H}{\alpha r H}
	=\frac{C M_0\varepsilon^\beta}{\alpha}r.
	\]
Therefore, we obtain
	\begin{align}
		\frac{H'}{H}
		\ge
		\frac{2\alpha+n-2}{r}+\frac{1}{\alpha r}\tilde{\mathcal{N}}(r)-\frac{P(r)}{\alpha r}-\frac{C M_0\varepsilon^\beta}{\alpha}r,\label{lower-1}\\
		\frac{H'}{H}
		\le
		\frac{2\alpha+n-2}{r}+\frac{1}{\alpha r}\tilde {\mathcal{N}}(r)-\frac{P(r)}{\alpha r}+\frac{C M_0\varepsilon^\beta}{\alpha}r
        \label{upper-1}.
	\end{align}
	
	Integrating the lower bound in \eqref{lower-1} from $2r_2$ to $r_3$ and using the monotonicity of  $\tilde{\mathcal{N}}(r)$ and the definition of $P(r)$, we get
	\begin{equation}\label{eq:lower-int}
		\begin{aligned}
			\log\frac{H(r_3)}{H(2r_2)} 
			&=\int_{2r_2}^{r_3}\frac{H'}{H}\,dr \\
			&\ge \int_{2r_2}^{r_3}\Bigl[\frac{2\alpha+n-2}{r}+\frac{1}{\alpha r}\tilde{\mathcal{N}}(2r_2)-\frac{P(r)}{\alpha r}-\frac{C M_0\varepsilon^\beta}{\alpha}r\Bigr]\,dr\\
			&=\Bigl[2\alpha+n-2+\frac{\tilde{\mathcal{N}}(2r_2)}{\alpha}\Bigr]\log\frac{r_3}{2r_2}
			-\Bigl(\frac{CM}{\alpha}+\frac{C M_0\varepsilon^\beta}{\alpha}\Bigr)r_3^2
			-\frac{CM_0\varepsilon^{\beta-1}}{\alpha}r_3^3
			-\frac{CM_0^2\varepsilon^{2\beta}}{\alpha^2}r_3^4.
		\end{aligned}
	\end{equation}
	Similarly, integrating the upper bound in \eqref{upper-1} from $r_1$ to $2r_2$ and using the monotonicity of $\tilde {\mathcal{N}}$, we obtain
	\begin{align}\label{eq:upper-int}
		\log\frac{H(2r_2)}{H(r_1)}
		&=\int_{r_1}^{2r_2}\frac{H'}{H}\,dr \nonumber \\
		&\le 
		\Bigl[2\alpha+n-2+\frac{\tilde{\mathcal{N}}(2r_2)}{\alpha}\Bigr]\log\frac{2r_2}{r_1}
		+\frac{C M_0\varepsilon^\beta}{2\alpha}(2r_2)^2.
	\end{align}
	Combining the inequality \eqref{eq:lower-int}--\eqref{eq:upper-int} yields that
		\begin{equation}\begin{aligned} \label{eq:bridge}
		\frac{\log\!\Bigl[\frac{H(2r_2)}{H(r_1)}\Bigr]-\frac{C K\varepsilon^\beta}{2\alpha}(2r_2)^2}{\log\!\Bigl[\frac{2r_2}{r_1}\Bigr]}
		&\le 2\alpha+n-2+\frac{\tilde{\mathcal{N}}(2r_2)}{\alpha} \\
		&\le \frac{\log\!\Bigl[\frac{H(r_3)}{H(2r_2)}\Bigr]+\Bigl(\frac{CM}{\alpha}+\frac{C M_0\varepsilon^\beta}{\alpha}\Bigr)r_3^2
			+\frac{CM_0\varepsilon^{\beta-1}}{\alpha}r_3^3
	+\frac{CM_0^2\varepsilon^{2\beta}}{\alpha^2}r_3^4}{\log\!\Bigl[\frac{r_3}{2r_2}\Bigr]}.
	\end{aligned}
    	\end{equation}
We want to get ride of the weight function $e_r(x)$. Let 	
$h(r):=\int_{B_r}u(x)^2\,dx.$
	Since $0\le e_r\le r^2$ on $B_r$, it is true that
	\begin{equation}\label{eq:Hlehr}
		H(r)=\int_{B_r}u^2 e_r^{\alpha-1}\le r^{2(\alpha-1)}h(r).
	\end{equation}
	We also have $e_\rho=\rho^2-|x|^2\ge \rho^2-r^2$ for $x\in B_r$ for $\rho>r$.  Hence
	\begin{equation}\label{eq:Hrhole}
		H(\rho)=\int_{B_\rho}u^2 e_\rho^{\alpha-1}\ge \int_{B_r}u^2 e_\rho^{\alpha-1}\ge (\rho^2-r^2)^{\alpha-1}h(r).
	\end{equation}
	Taking $\rho=2r_2$ and $r=r_2$ gives that
	\[
	H(2r_2)\ge (3r_2^2)^{\alpha-1}h(r_2)=\Bigl(\frac{3}{4}\Bigr)^{\alpha-1}(2r_2)^{2(\alpha-1)}h(r_2).
	\]
	From \eqref{eq:Hlehr} and \eqref{eq:Hrhole},  we obtain that
	\begin{align}
		\log\frac{H(r_3)}{H(2r_2)}
		&\le \log\frac{h(r_3)}{h(r_2)}+(\alpha-1)\Bigl[2\log\frac{r_3}{2r_2}+\log\frac{4}{3}\Bigr],\label{eq:log1}\\
		\log\frac{H(2r_2)}{H(r_1)}
		&\ge \log\frac{h(r_2)}{h(r_1)}+(\alpha-1)\Bigl[2\log\frac{2r_2}{r_1}-\log\frac{4}{3}\Bigr].\label{eq:log2}
	\end{align}
	Inserting the estimates \eqref{eq:log1}--\eqref{eq:log2} into \eqref{eq:bridge}, we
	 arrive at the three-ball inequality
	\begin{align}\label{eq:threeball-h}
		h(r_2)&\le
		\exp\left\{
		(\alpha-1)\log\frac{4}{3}
		+\frac{C M_0\varepsilon^\beta}{\alpha}R^2
		+\frac{M}{\alpha}R^2
		+\frac{C M_0\varepsilon^{\beta-1}}{\alpha}R^3
		+\frac{C M_0^2\varepsilon^{2\beta}}{\alpha^2}R^4
		\right\}\nonumber \\
        &\quad  \times h^{\frac{\gamma_2}{\gamma_1+\gamma_2}}(r_1)\,h^{\frac{\gamma_1}{\gamma_1+\gamma_2}}(r_3),
	\end{align}
    where \[
	\gamma_1:=\log\frac{2r_2}{r_1},
	\qquad
	\gamma_2:=\log\frac{r_3}{2r_2}.
	\]

    Since $\|V\|_{C^{0, \beta}}=\max\{ M, M_0\}\geq 1$. Without loss of generality,  let  $M_0\leq M$. Assume $R>1$. It is obvious that 
$\frac{C M \varepsilon^\beta}{\alpha}R^2
		<\frac{M}{\alpha}R^2
		<\frac{C M \varepsilon^{\beta-1}}{\alpha}R^3$
    as $0<\beta<1$ and $\varepsilon$ is small.
 We choose
$ \varepsilon= \Bigl(\frac{\alpha}{MR}\Bigr)^{\frac{1}{\beta+1}}$.
	%then
	%\[
	%\varepsilon^{\beta-1}
	%=\min\left\{\Bigl(\frac{R}{2}\Bigr)^{\beta-1},\ \Bigl(\frac{\alpha}{MR}\Bigr)^{\frac{\beta-1}{\beta+1}}\right\},
	%\qquad
	%\varepsilon^{2\beta}\le \Bigl(\frac{\alpha}{MR}\Bigr)^{\frac{2\beta}{\beta+1}}.
	%\]
	Thus,
	\[
	\frac{CMR^3}{\alpha}\,\varepsilon^{\beta-1}
	\le C R^{\frac{4+2\beta}{\beta+1}}\,M^{\frac{2}{\beta+1}}\,\alpha^{-\frac{2}{\beta+1}},
	\] 
	%CMR^3\alpha^{-1}\Bigl(\frac{MR}{2}\Bigr)^{\frac{1-\beta}{1+\beta}}
	and
	\[
	\frac{CM^2R^4}{\alpha^2}\,\varepsilon^{2\beta}
	\le
	%CM^2R^4\alpha^{-2}\Bigl(\frac{\alpha}{MR}\Bigr)^{\frac{2\beta}{\beta+1}}
	%=
	C R^{\frac{4+2\beta}{\beta+1}}\,M^{\frac{2}{\beta+1}}\,\alpha^{-\frac{2}{\beta+1}}.
	\]
	It is reduced to to minimize the exponent $\alpha+C R^{\frac{4+2\beta}{\beta+1}}\,M^{\frac{2}{\beta+1}}\,\alpha^{-\frac{2}{\beta+1}}$ in the coefficient of the three-ball inequality (\ref{eq:threeball-h}).
	%Since $\frac{MR^2}{\alpha}\leq C  R^{\frac{4+2\beta}{\beta+1}}\,M^{\frac{2}{\beta+1}}\,\alpha^{-\frac{2}{\beta+1}}$ from the fact that $0<\beta\leq 1$, 
    In order to achieving the minimum, we take 
$\alpha = M^{\frac{2}{\beta+3}}\,R^{\frac{4+2\beta}{\beta+3}}.$
	Therefore, the three-ball inequality with the minimum exponent in (\ref{eq:threeball-h}) is given as 
    	\begin{align*}
		h(r_2)\le
		\exp\left\{CM^{\frac{2}{\beta+3}}R^{\frac{4+2\beta}{\beta+3}}
		\right \}
h(r_1)^{\theta}\,h(r_3)^{1-\theta},
	\end{align*}
    where $\theta=\frac{\gamma_2}{\gamma_1+\gamma_2}=\frac{\log \frac{r_3}{2r_2}}{\log \frac{r_3}{r_1} }, $
     which completes the proof of the Theorem.
	\end{proof}

    The vanishing order of solutions follows from the quantitative three-ball inequality.

      \begin{proof}[Proof of Corollary \ref{thm:vanishing-I-correct}]
Let $r_1=r, r_2=1, r_3=4$ and $R=10$
    	in (\ref{eq:threeballs-I-final}). 
    From the assumption of $u$ in (\ref{eq:assump-norm}), we obtain that
    	\[
    	1\le \|u\|_{L^2(B_r)}^{\theta}\,\bigl (C_0^2|B_{4}|\bigr )^{1-\theta}\,
    	\exp\!\left\{\,C\,M^{\frac{2}{\beta+3}}\,10^{\frac{4+2\beta}{\beta+3}}\right\}.
    	\]
    	Since $\theta^{-1}=\frac{\log 4-\log r}{\log 2}$, we can show that \begin{align}
    	    \|u\|_{L^2(B_r)}\ge C\,r^{C(1+M^{\frac{2}{\beta+3}})},
    	\end{align}
        which implies the upper bound  $CM^{\frac{2}{\beta+3}}$ for the vanishing order at origin. Since the three-ball inequality is translation invariant, we can show the same upper bound of vanishing order in $B_{1/2}.$
    \end{proof}

\section{Three-ball inequality for elliptic equations with drift term}\label{nabla term and lifting}

    In this section, we study the quantitative unique continuation for elliptic equations with drift term,\begin{equation}\label{eq:mainPDE}
    	-\Delta u + W\cdot \nabla u = 0 \qquad \text{in } B_R\subset\mathbb{R}^n,
    \end{equation}
    where $W\in C^{0,\beta_0}(B_R;\mathbb{R}^n)$ is H\"older continuous.
    Usually, it is more challenging to study unique continuation or quantitative unique continuation for the gradient term $W(x)$. 
    %See the study of  strong unique continuation property for critical Lebesgue space of the gradient term cite{W}, {T}, which appears to have more difficulty  than the study of critical Lebesgue space for the potentials. 
    We can not find the sharp growth power for the solutions in (\ref{eq:mainPDE}) using the weighted frequency function directly as (\ref{frequency-2}) in the last section. Instead, we adopt  a lifting argument to incorporate the regular parts of gradient term into the leading coefficients. In the lifting arguments, we need to make use of the second order derivative of $W_\varepsilon$, which is different from the section 2 where the first order derivative of $V_\varepsilon$ is  used.
    
  We still do regularization argument for $W(x)$.
 Let $W_\varepsilon(x)=\rho_\varepsilon*W$ be the standard mollification. Then
    $W_\varepsilon\in C^\infty(B_{R-\varepsilon})$ and the equation \eqref{eq:mainPDE} can be written as
    \begin{equation}\label{eq:mollifiedPDE}
    	-\Delta u + W_\varepsilon\cdot \nabla u = g_\varepsilon\cdot \nabla u
    	\qquad\text{in }B_{R-\varepsilon},
    \end{equation}
    where $g_\varepsilon:=W_\varepsilon-W$.
    Assume
    $\|W\|_{C^{0,\beta_0}(B_R)}\le K$.
   As the arguments in Lemma \ref{lem:moll},   it holds that 
    \begin{equation}\label{eq:mollEst}
    	|W_\varepsilon|\le K,\qquad
    	|\nabla W_\varepsilon|\le CK\,\varepsilon^{\beta_0-1},\qquad
        |\nabla^2 W_\varepsilon|\le CK\,\varepsilon^{\beta_0-2},\qquad
    	|g_\varepsilon|\le CK\,\varepsilon^{\beta_0}
    \end{equation}
in $B_{R-\varepsilon}$ for a constant $C(n)$. 
    Set the vector $I=(1,\dots,1)\in\mathbb{R}^n$ and
    \begin{align}\label{BBB-1}
    B:=\|W_\varepsilon\|_{C^1(B_{R-\varepsilon})}+1.
    \end{align}
We rewrite \eqref{eq:mollifiedPDE} in the following way
\begin{equation}\label{eq:firstRewrite}
    	-\Delta u
+\Bigg[\frac{W_\varepsilon+\bigl(\|W_\varepsilon\|_{C^1}+1\bigr)I}{LB}\Bigg]\cdot \nabla(LBu)
    	=
    	g_\varepsilon\cdot \nabla u + B\,I\cdot \nabla u,
    \end{equation}
    where $L$ is a large positive constant to be  determined later.
    Denote
    \begin{equation}\label{eq:defb}
    	b:=\frac{W_\varepsilon+\bigl(\|W_\varepsilon\|_{C^1}+1\bigr)I}{LB},
    	\qquad
    	\hat{u}(x,t):=e^{LBt}\,u(x).
    \end{equation}
     By the fact that $\partial_{tt}\hat{u}=L^2B^2\hat{u}$, the equation \eqref{eq:firstRewrite} turns into
    \begin{equation}\label{eq:Veq1}
    	-\partial_{tt}\hat{u}-\Delta_x \hat{u} + b\cdot \nabla_x(\partial_t \hat{u})
    	=
    	g_\varepsilon\cdot \nabla_x \hat{u} + B\,I\cdot \nabla_x \hat{u} - L^2B^2\hat{u}.
    \end{equation}
    
The left hand side of the last equation can be written as
    \begin{equation}\label{eq:divForm1}
    	-\partial_{tt}\hat{u}-\Delta_x \hat{u} + b\cdot \nabla_x(\partial_t \hat{u})
    	=
    	-\operatorname{div}_{x,t}\!\bigl(A_1\nabla_{x,t}\hat{u}\bigr)
    	-\frac{{\rm div}_x b}{2}\,\partial_t \hat{u},
    \end{equation}
    where the $(n+1)\times(n+1)$ matrix $A_1$ is given as
    \begin{equation}\label{eq:A1def}
    	A_1:=
    	\begin{pmatrix}
    		I_{n\times n} & -\frac{b}{2}^T\\
    		-\frac{b}{2} & 1
    	\end{pmatrix}.
    \end{equation}
    Indeed, we can show that $$-{\rm div}_{x,t}(A_1\nabla_{x,t}\hat{u})=-\partial_{tt}\hat{u}-\Delta_x\hat{u}+ b\cdot\nabla_x(\partial_t\hat{u})
    +\frac{{\rm div}_x b}{2}\partial_t\hat{u}.$$
    We want $A_1$ to be uniformly elliptic.
    For $\xi=(\xi_1,\dots,\xi_{n+1})\in\mathbb{R}^{n+1}$,
    \begin{align}\label{eq:A1ell}
    	\xi\cdot A_1\,\xi^T
    	&=\xi_1^2+\cdots+\xi_{n+1}^2-\sum_{k=1}^n b_k\,\xi_k\,\xi_{n+1}\\
    	&\ge |\xi|^2-\Bigl(\sum_{k=1}^n b_k^2\Bigr)^{1/2}
    	\Bigl(\sum_{k=1}^n \xi_k^2\Bigr)^{1/2}\xi_{n+1}\notag\\
    	&\ge |\xi|^2-|b|\big(  \sum_{k=1}^n\xi_k^2+\xi_{n+1}^2  \big)
    	=\Bigl(1-|b|\Bigr)|\xi|^2.\notag
    \end{align}
  By (\ref{BBB-1}) and \eqref{eq:defb}, 
    \begin{equation}\label{eq:bbound}
    	|b|
    	\le
\frac{\|W_\varepsilon\|_{C^1}+\bigl(\|W_\varepsilon\|_{C^1}+1\bigr)\sqrt{n}}
    	{LB}
    	\le \frac{\sqrt{n}+1}{L}.
    \end{equation}
    Taking $L$ to be so large that $|b|\le \frac{1}{2}$,  by \eqref{eq:A1ell},  we obtain that
    \[
    \xi\cdot A_1 \,\xi^T \ge \frac12|\xi|^2.
    \]
    Thus, $A_1$ is uniformly elliptic. Moreover, we can have $|\nabla a_{ij}|\le C|\nabla b|\le 1$ by  taking $L$ large. 
    
  Since  $\partial_t\hat{u}=LB\hat{u}$, it follows from the equation \eqref{eq:Veq1} that
    \begin{equation}\label{eq:VeqDiv}
-\operatorname{div}_{x,t}\!\bigl(A_1\nabla_{x,t}\hat{u}\bigr)
    	=
    	g_{1,\varepsilon}\cdot \nabla_{x,t}\hat{u}
    	+ B\,I_1\cdot \nabla_{x,t}\hat{u}
    	- L^2B^2\hat{u}
    	+\frac{{\rm div}_x b}{2}\,LB\hat{u},
    \end{equation}
    where we use the natural embedding
    \[
    g_{1,\varepsilon}:=(g_\varepsilon,0)\in\mathbb{R}^{n+1},
    \qquad
    I_1:=(I,0)\in\mathbb{R}^{n+1}.
    \]
    
   We perform further lifting arguments to absorb $B\,I_1\cdot \nabla_{x,t}\hat{u}$ into the leading coefficients. We write
    \[
    B\,I_1\cdot \nabla_{x,t}\hat{u}
    =
    \frac{B\,I_1}{LB}\cdot \nabla_{x,t}(LB\hat{u}).
    \]
    Let
    \[
    \widehat b:=\frac{I_1}{L},
    \qquad
    \bar{u}(x,t,s):=e^{LBs}\,\hat{u}(x,t).
    \]
    Then $\partial_s \bar{u}=LB\bar{u}$ and the same lifting process as above gives the equation for $\bar{u}$:
    \begin{equation}\label{eq:WeqDiv}
    	-\operatorname{div}_{x,t,s}\!\bigl(A_2\nabla_{x,t,s}\bar{u}\bigr)
    	=
    	g_{2,\varepsilon}\cdot \nabla_{x,t,s}\bar{u}
    	-2L^2B^2\bar{u}
    	+\frac{{\rm div}_x b}{2}\,LB\bar{u},
    \end{equation}
    where we use $\partial_{ss}\bar u=L^2B^2 \bar u$, $ g_{2,\varepsilon}=(g_{1,\varepsilon},0)\in\mathbb{R}^{n+2},$
    and the $(n+2)\times(n+2)$ matrix $A_2$ takes the form
    \begin{equation}\label{eq:A2def}
    	A_2:=
    	\begin{pmatrix}
    		A_1 &  \frac{\widehat b}{2}^T\\
    		\frac{\widehat{b}}{2} & 1
    	\end{pmatrix}.
    \end{equation}
    We can show $A_2$ is uniformly elliptic by taking $L$ large.

   \medskip

  \medskip
    Next we incorporate the term $-2L^2B^2\bar{u}
    	+\frac{{\rm div}_x b}{2}\,LB\bar{u}$ in (\ref{eq:WeqDiv}) into the leading coefficients.
%\[\label{equ:second lifting} 
%-%\operatorname{div}_{x,t,s}\!\bigl(A_2\nabla_{x,t,s}\bar u\bigr)
%= f_{2,\varepsilon}\nabla_{x,t,s}\bar u
%-2L^2B^2\bar u
%+\frac{{\rm div}_x b}{2} LB\,\bar u.
%\]
Denote $\hat{B}:=\|W_\varepsilon\|_{C^2}+1$.
From \eqref{eq:mollEst}, we have $\hat{B}\le CK\varepsilon^{\beta-2}.$
It holds that
\[
\bigl(-2L^2B^2+(\operatorname{div}_x b)LB/2\bigr)\bar u
=
\frac{(\operatorname{div}_x b)LB/2+L_1\hat{B}}{L_1^2 {\hat{B}}}
\Bigl(L_1^2 {\hat{B}}\bar u\Bigr)-\bigl(L_1\hat{B}+2L^2B^2\bigl)\bar{u}.
\]
We introduce 
\[
\widetilde{u}(x,t,s,y)
=
e^{\,L_1 \sqrt{\hat{B}}y}\,\bar u(x,t,s),
\]
where \(L_1\) is a large positive constant to be determined. Then
\[
\partial_{yy}\widetilde{u}
=
L_1^2 {\hat{B}}\widetilde{u},
\]
and
\[
\bigl(-2L^2B^2+\,(\operatorname{div}_x b)LB/2\bigr)\widetilde{u}
=
\frac{(\operatorname{div}_x b)LB/2+L_1\hat{B}}{L_1^2 {\hat{B}}}
\partial_{yy}^2\widetilde{u}-\bigl(L_1\hat{B}+2L^2B^2\bigl)\widetilde{u}.
\]
Thus it follows from \eqref{eq:WeqDiv} that
\begin{equation}\label{eq: Final lifting}
    -\operatorname{div}_{x,t,s,y}\!\bigl(A_3\nabla_{x,t,s,y}\widetilde{u}\bigr)
=
g_{3,\varepsilon}\nabla_{x,t,s,y}\widetilde{u}-(2L^2B^2+L_1\hat{B})\widetilde{u},
\end{equation}
where
\[
A_3:=
\begin{pmatrix}
A_2 & 0\\[1ex]
0 & \dfrac{L_1\hat{B}+(\operatorname{div}_x b)LB/2}{L_1^2 {\hat{B}}}
\end{pmatrix}.
\]
Note that $\operatorname{div}_x b = \frac{\operatorname{div}_x W_{\varepsilon}}{LB}, $
%\begin{equation*}
%    \operatorname{div}_x b = \frac{\operatorname{div}_x W_{\varepsilon}}{LB}, \qquad |\operatorname{div}_x b| \leq C\frac{K\varepsilon^{\beta_0 - 1}}{LB}.
%\end{equation*}
Denote
\begin{equation*}
    a_{yy}: = \frac{L_1\hat{B}+\operatorname{div}_x W_\varepsilon/2}{L_1^2\hat{B}}. %= {\color{red} \frac{L_1\hat{B}-\operatorname{div}_x W_\varepsilon/(2L^2B^2)}{L_1^2\hat{B}}}
\end{equation*}
If $L_1$ is large enough, we have
\begin{equation*}
    \frac{1}{2L_1} \leq a_{yy} \leq 1, \qquad |\nabla a_{yy}| \leq \frac{C|\nabla^2 W_{\varepsilon}|}{L_1^2 \hat{B}} \leq C,
\end{equation*}
where we used the second derivative of $W_\varepsilon$ in (\ref{eq:mollEst}).
Then $A_3$ is uniformly elliptic and $|\nabla A_3| \leq C_1$ for some universal constant $C_1$.

%Now we take \(L_1\) large such that
%\[
%|\nabla A_3|\le \Lambda
%\]
%for a universal constant $\Lambda$ and that \(A_3\) is uniformly elliptic:
%\[
%A_3x\cdot x\ge \lambda |x|^2,\qquad \forall x\in \mathbb{R}^{n+3}.
%\]

Finally, we absorb the term $-(2L^2B^2+L_1B)\widetilde{u}$ in (\ref{eq: Final lifting}).
%Using %$(2L^2B^2+L_1\hat{B}) %\widetilde{u}=\frac{2L^2B^2%+L_1B}{L_1^2L^2B^2}%(L_1^2L^2B^2\widetilde{u})$,
Let $\widetilde{\bar u}(x,t,s,y,z)=e^{i\sqrt{2L^2B^2+L_1\hat{B}}z}\widetilde{u}(x,t,s,y)$. Then we can write \eqref{eq: Final lifting} to be
\begin{equation}\label{eq: real Final lifting}
    -\operatorname{div}_{x,t,s,y,z}\!\bigl(A_4\nabla_{x,t,s,y,z}\widetilde{\bar u}\bigr)=g_{4,\varepsilon}\nabla_{x,t,s,y,z}\widetilde{\bar u},
\end{equation}
where 
\[
A_4:=
\begin{pmatrix}
A_3 & 0\\[1ex]
0 & 1
\end{pmatrix},
\]
and $g_{4, \varepsilon}:=(g_{3,\varepsilon},0)\in \mathbb{R}^{n+4}$.
Again $A_4$ is uniformly elliptic and has uniform Lipschitz upper bound.
%\textcolor{red}{
 %  Finally, let
%    \begin{equation*}
%    \widetilde{u}(x,t,s,y) = e^{iLBy} \bar u(x,t,s),
%\end{equation*}
%then
%\begin{equation*}
%    \partial_{yy}^2 \widetilde{u} = -L^2B^2 \widetilde{u}
%\end{equation*}
%and
%\begin{equation*}
%    \bigl(-2L^2B^2+\,(\operatorname{div}_x b)LB/2\bigr)\widetilde{u} = \frac{-2L^2B^2 + (\operatorname{div}_xb) LB/2}{-L^2B^2} \partial_{yy}^2 \widetilde{u}.
%\end{equation*}
%Thus it follows from \eqref{eq:WeqDiv} that
%\begin{equation}\label{eq:Final lifting}
 %   -%\operatorname{div}_{x,t,s,y}\!\bigl(A_3\nabla_{x,t,s,y}\widetilde{u}\bigr)
%=
%g_{3,\varepsilon}\nabla_{x,t,s,y}\widetilde{u}
%\end{equation}
%where
%\[
%A_3=
%\begin{pmatrix}
%A_2 & 0\\[1ex]
%0 & \dfrac{(\operatorname{div}_x b)LB/2-2L^2B^2}{-L^2B^2}
%\end{pmatrix},
%\]
%Note that
%\begin{equation*}
 %   \operatorname{div}_x b = \frac{\operatorname{div}_x W_{\varepsilon}}{LB}, \qquad |\operatorname{div}_x b| \leq C\frac{K\varepsilon^{\beta_0 - 1}}{LB}.
%\end{equation*}
%Denote
%\begin{equation*}
%    a_{yy} = \frac{(\operatorname{div}_x b)LB/2-2L^2B^2}{-L^2B^2} = 2 - \frac{\operatorname{div}_x W_{\varepsilon}}{2L^2B^2}.
%\end{equation*}
%If $L$ is large enough and we will assume that $K\varepsilon^{\beta_0} > 1$, we have
%\begin{equation*}
 %   \frac{3}{2} \leq |a_{yy}| \leq \frac{5}%{2}, \qquad |\nabla a_{yy}| \leq \frac{C|\nabla^2 W_{\varepsilon}|}{2L^2B^2} \leq \frac{C}{L^2 K \varepsilon^{\beta_0}}.
%\end{equation*}
%Then $A_3$ is uniform elliptic and $|\nabla A_3| \leq C_1$ for some universal constant $C_1$.}

Applying Proposition 4.1 in \cite{D25} to \eqref{eq: real Final lifting} to have
\begin{equation}\label{ddd-1}
\|\widetilde{\bar u}\|_{r_2}
\le
{e^{C R}}
\exp\!\left\{
\bigl[CK^2\varepsilon^{2\beta_0}R^2+C\bigr]
\bigl[1+\log \tfrac{r_3}{2r_2}\bigr]
\right\}
\|\widetilde{\bar u}\|_{r_1}^{\theta}
\|\widetilde{\bar u}\|_{r_3}^{1-\theta},
\end{equation}
where we used the fact that $\|g_{4, \varepsilon}\|_{L^\infty}\leq CK \varepsilon^{\beta_0}$,
$\|\widetilde{ \bar u}\|_r:=\|\widetilde{ \bar u}\|_{L^2(B_r)}$ for
$B_r\subset \mathbb R^{n+4}$ and 
\begin{align}
\theta=
\frac{\log r_3-\log 2r_2}
{\log r_3-\log 2r_2+{Ce^{C(r_3-r_1)}}\bigl[\log 2r_2-\log r_1\bigr]},
\label{theta-1}
\end{align}
with $C$ depending on $n$. {Note that we choose a slight different $\theta$ from \cite{D25}. This $\theta$ can also be seen from the arguments for elliptic equations with  variable leading coefficients in next section.}
Recall all the lifting arguments we have performed. It follows that
\[
\widetilde{\bar u}
=
e^{LB(s+t)+ L_1\sqrt{\hat B}y +i\sqrt{2L^2B^2+L_1\hat{B}}z}u(x).
\]
Then
\begin{equation}
 \begin{aligned}
  	& \|\widetilde{\bar u}\|_{r_2}^2=\int_{B^{x,t,s,y,z}_{r_2}} \widetilde{\bar u}^2
  	\ge e^{-4r_2LB - 2r_2L_1\sqrt{\hat{B}}}\int_{B_{r_2}} u^2(x)\,dx, \\
  	&\|\widetilde{\bar u}\|_{r_1}^2\le e^{4r_1LB + 2r_1L_1\sqrt{\hat{B}}}\int_{B_{r_1}} u^2(x)\,dx,\\
  	&\|\widetilde{\bar u}\|_{r_3}^2\le e^{4r_3LB + 2r_3L_1\sqrt{\hat{B}}}\int_{B_{r_3}} u^2(x)\,dx.
  \end{aligned}
  \label{compare-1}
  \end{equation}
  Through (\ref{ddd-1}, (\ref{compare-1}) and estimates  of $B$, $\hat{B}$ in (\ref{eq:mollEst}),
we obtain the quantitative three-ball inequality for $u$ in $\mathbb R^n$,
\begin{equation}
\|u\|_{r_2}
\le
{e^{CR}}
\exp\!\left\{
CRK\varepsilon^{\beta_0-1}
+ CRK^{1/2}\varepsilon^{\beta_0/2 - 1}+
\bigl[CK^2\varepsilon^{2\beta_0}R^2+C\bigr]
\bigl[1+\log \tfrac{r_3}{2r_2}\bigr]
\right\}
\|u\|_{r_1}^{\theta}\|u\|_{r_3}^{1-\theta}.
\end{equation}
%where
%\[
%\|u\|_r:=\|u\|_{L^2(B_r^x)}.
%\]
We can choose $R\geq 1$ and $K\varepsilon^{\beta_0}>1$. It is obvious that $RK^\frac{1}{2}\varepsilon^{\frac{\beta_0}{2}-1}\leq RK\varepsilon^{{\beta_0}-1}$.
To minimize the exponent in the last inequality, we select
\[
K\varepsilon^{\beta_0-1}=K^2\varepsilon^{2\beta_0} R.
\]
Then $\varepsilon=(RK)^{-\frac1{\beta_0+1}}$.
%and in this choice
%\[
%K\varepsilon^{\beta-1}=K^2\varepsilon^{2\beta}=K^{\frac{2}{\beta+1}},
%\]
%\[K^{1/2}\varepsilon^{\beta/2-1}
%=
%K^{\frac{3}{2(\beta+1)}}<K^{\frac{2}{\beta+1}}.
%\]
Since $0<\beta_0<1$,  we obtain the three-ball inequality
\begin{equation}\label{eq:3ballsuFinal}
    	\|u\|_{r_2}
    	\le
    	\exp\!\Bigg\{
    	CK^{\frac{2}{\beta_0+1}}R^{\frac{2}{\beta_0+1}}
    	\Bigl[
    	1+\log\!\Bigl(\frac{r_3}{2 r_2}\Bigr)
    	\Bigr]
    	\Bigg\}
    	\,
    	\|u\|_{r_1}^{\theta}\,
    	\|u\|_{r_3}^{1-\theta}
    \end{equation}
   with $\theta$ is given in (\ref{theta-1}). This completes the proof of Theorem \ref{thm:threeballs-III-final}.
\medskip

Note that the three-ball inequality in Theorem (\ref{thm:threeballs-I-opt-correct}) is rescaling invariant.
We can also prove a rescaling invariant three-ball inequality for the solutions in \eqref{eq:mainPDE-1}.
 \begin{corollary}
 \label{add-on}
 	Let $0<\beta_0< 1$ and $u$ solve (\ref{eq:mainPDE-1}) with $R\geq 1$.
					Assume $\|W\|_{C^{0,\beta_0}}\leq K$.
					Then there exists some  constants $C(n, \beta_0)$ such that for any $0<r_1<r_2<2r_2<r_3<\frac{R}{2}$,
\begin{equation}
					\|u\|_{r_2}
    	\le
    	\exp\!\Bigg\{
    	CK^{\frac{2}{\beta_0+1}}R^{{2}}
    	\Bigl[
    	1+\log\!\Bigl(\frac{r_3}{2 r_2}\Bigr)
    	\Bigr]
    	\Bigg\}
    	\,
    	\|u\|_{r_1}^{\theta}\,
    	\|u\|_{r_3}^{1-\theta},
					\end{equation}
					where $0<\theta=\frac{\log r_3-\log 2r_2}
{\log r_3-\log 2r_2+{C}(\log 2r_2-\log r_1)}<1.$
    				     
    				 \end{corollary}
\begin{proof}
Let $R=1$ in Theorem \ref{thm:threeballs-III-final}. The estimate (\ref{eq:threeballs-III-final}) implies the following three-ball inequality
\begin{equation}
					\|u\|_{r_2}
    	\le
    	\exp\!\Bigg\{
    	CK^{\frac{2}{\beta_0+1}}
    	\Bigl[
    	1+\log\!\Bigl(\frac{r_3}{2 r_2}\Bigr)
    	\Bigr]
    	\Bigg\}
    	\,
    	\|u\|_{r_1}^{\tilde{\theta}}\,
    	\|u\|_{r_3}^{1-\tilde{\theta}},
        \label{rescaling-100}
					\end{equation}
where $0<\tilde{\theta}:=\frac{\log r_3-\log 2r_2}
{\log r_3-\log 2r_2+{C_1}(\log 2r_2-\log r_1)}<\theta
=\frac{\log r_3-\log 2r_2}
{\log r_3-\log 2r_2+{C} e^{C(r_3-r_1)}(\log 2r_2-\log r_1)}<1$. We perform a rescaling argument. Let $v(x)=u(Rx)$. Then $v(x)$ satisfies 
\begin{align}
  \triangle v(x)+R\tilde{W}(x) \nabla v(x)=0\quad \mbox{in} \ B_1. 
  \label{111-v}
\end{align}
where $\tilde{W}(x):=W(Rx)$. If $\|W\|_{C^{0,\beta_0}}\leq K$, then $\|R\tilde{W}\|_{C^{0, \beta_0}}\leq K R^{1+\beta_0}$. Applying the estimates 
(\ref{rescaling-100}) to $v$ yields that
    	\begin{align}		
                \|v\|_{r_2}
    	\le
    	\exp\!\Bigg\{
    	CK^{\frac{2}{\beta_0+1}}R^2
    	\Bigl[
    	1+\log\!\Bigl(\frac{r_3}{2 r_2}\Bigr)
    	\Bigr]
    	\Bigg\}
    	\,
    	\|v\|_{r_1}^{\tilde{\theta}}\,
    	\|v\|_{r_3}^{1-\tilde{\theta}}.
        \label{rescaling-99}
					\end{align}
                    Since $\tilde{\theta} $ and $\log \frac{r_3}{2 r_2}$ are rescaling invariant values, we substitute $v(x)=u(Rx)$ into the last inequality to obtain
                    \begin{align}
    	\|u\|_{r_2}
    	\le
    	\exp\!\Bigg\{
    	CK^{\frac{2}{\beta_0+1}}R^{{2}}
    	\Bigl[
    	1+\log\!\Bigl(\frac{r_3}{2 r_2}\Bigr)
    	\Bigr]
    	\Bigg\}
    	\,
    	\|u\|_{r_1}^{\tilde{\theta}}\,
    	\|u\|_{r_3}^{1-\tilde{\theta}}.
					\end{align}         This finishes the proof of the Corollary.       
\end{proof}

The three-ball inequality in \ref{thm:threeballs-III-final} and Corollary \ref{add-on} can show the vanishing order estimates.
      \begin{proof}[Proof of Corollary \ref{thm:vanishing-III-correct}]
Let $r_1=r, r_2=1, r_3=4$ and $R=10$ in \eqref{eq:3ballsuFinal}. 
 By the normalized assumption for $u$ in \eqref{eq:assump-norm}, we can show that
    	\[
    	1\le \|u\|_{L^2(B_r)}^{\theta}\,(C_0^2|B_{4}|)^{1-\theta}\,
    	\exp\!\left\{\,C\,K^{\frac{2}{\beta_0+1}}\,10^{\frac{2}{\beta_0+1}}\right\}.
    	\]
    	It follows from (\ref{theta-1})
        that $\theta^{-1}=\frac{\log 2+C\log 2-C\log r}{\log 2}$. Hence it holds that\begin{align}
    	    \|u\|_{L^2(B_r)}\ge C\,r^{C(1+K^{\frac{2}{\beta_0+1}})}.
    	\end{align}
        This implies the upper bound  $CK^{\frac{2}{\beta_0+1}}$  for the vanishing order of $u$ at origin. Since the three-ball inequality is translation invariant,  the same upper bound of vanishing order is arrived in $B_{1/2}.$
    \end{proof}

  %  \begin{proof}
    %	Apply \eqref{eq:threeballs-III-final} with $(r_1,r_2,r_3,R,\sigma)=(r,1,10,10,2)$:
    %	\[
    %	\|u\|_{L^2(B_1)}
    %	\le \mathcal{C}\,\|u\|_{L^2(B_r)}^{\vartheta}\,\|u\|_{L^2(B_{10})}^{1-\vartheta},
    %	\]
    %	where $\mathcal{C}$ is the explicit constant from \eqref{eq:threeballs-III-final}
    %	(with these fixed radii) and $\vartheta$ is given by \eqref{eq:vartheta-final}.
    %	Using $\|u\|_{L^2(B_1)}\ge 1$ and $\|u\|_{L^2(B_{10})}\le C_0|B_{10}|^{1/2}$ yields
    	%\[
    %	\|u\|_{L^2(B_r)}\ge \mathcal{C}^{-1/\vartheta}(C_0|B_{10}|^{1/2})^{-(1-\vartheta)/\vartheta}.
    %	\]
    %The same argument as above implies the result.
   
   % \end{proof}
    	\section{Three-ball inequality for elliptic equation with variable coefficients }\label{general eqution}

       % Variable coefficient equation $-\partial_j\!\bigl(a_{ij}(x)\partial_i u\bigr)+V(x)u=0$
    	In this section, we study 
the quantitative unique continuation for 
    	\begin{equation}\label{eq:PDE}
    		-\partial_j\!\bigl(a_{ij}(x)\partial_i u\bigr)+V(x)u=0\qquad \text{in }B_R.
    	\end{equation}
The summation for the index $i, j$ is understood by the Einstein summation convention. The main ideas follow from the arguments in section 1. There are calculation technicalities to deal with variable leading coefficients.
    	Assume $A(x)=(a_{ij}(x))_{1\le i,j\le n}$ be a symmetric matrix on $B_R\subset\mathbb{R}^n$ which satisfies \eqref{lips-1}.
      % Z% such that
%\begin{equation}%\label{eq:ellipticity}
    	%	\lambda |\xi|^2 \le a_{ij}(x)\xi_i\xi_j \le \Lambda |\xi|^2,
    		%\qquad \forall \xi\in\mathbb{R}^n,\ \forall x\in B_R,
    
   % 	and $A$ is Lipschitz:
    	%\begin{equation}\label{eq:LipA}
    		%\|\nabla A\|_{L^\infty(B_R)} := \max_{i,j}\|\nabla a_{ij}\|_{L^\infty(B_R)} \le L_A.
    %	\end{equation}
    	Let $V\in C^{0,\beta}(B_R)$ with $0<\beta< 1$ and \begin{equation}\label{eq:Vbounds}
    		\|V\|_{L^\infty(B_R)}\le M,\qquad [V]_{C^{0,\beta}(B_R)}\le M_0.
    	\end{equation}
   As in section 1, we want to see the role of the H\"older semi-norm $[V]_{C^{0,\beta}}$ in the estimates.
    	Fix a point $0\in B_R$. After an invertible linear change of variables
    	(depending only on $A(0)$), we may assume $A(0)=I_{n\times n}.$
    	Then by \eqref{lips-1}
\begin{equation}\label{eq:Aclose}
    		\|A(x)-I_{n\times n}\|\le L_A r
    	\end{equation}
 for $|x|\le r$.
    We perform the same mollification for $V(x)$ as Lemma \ref{lem:moll}.
    	%Let $\rho\in C_0^\infty(B_1)$ be a standard mollifier, $\rho\ge0$, $\int_{B_1}\rho=1$,
    	%and set $\rho_\varepsilon(x)=\varepsilon^{-n}\rho(x/\varepsilon)$, $V_\varepsilon=\rho_\varepsilon*V$,
    	%$f_\varepsilon:=V_\varepsilon-V$. Then the standard estimates \label{lem:mollify}
        Thus,
    		for every $0<\varepsilon\le 1$,\begin{align}\label{VVV-1}
\|V_\varepsilon\|_{L^\infty(\mathbb{R}^n)}\le M,
    		\qquad
    		\|\nabla V_\varepsilon\|_{L^\infty(\mathbb{R}^n)}\le C M_0\varepsilon^{\beta-1},
    		\qquad
\|f_\varepsilon\|_{L^\infty(\mathbb{R}^n)}\le C M_0 \varepsilon^\beta,
    	    	\end{align}
    		where $f_\varepsilon:=V_\varepsilon-V$ and  $C=C(n)$.
    	With these, wer rewrite \eqref{eq:PDE}  as
    	\begin{equation}\label{eq:PDEeps}
    -\partial_j\!\bigl(a_{ij}\partial_i u\bigr)+V_\varepsilon u=f_\varepsilon u\qquad \text{in }B_R.
    	\end{equation}
    	Define the coefficient-dependent weight 
    	\begin{equation}\label{eq:mu}
    		\nu(x):=\frac{A(x)x\cdot x}{|x|^2}=\frac{a_{ij}(x)x_ix_j}{|x|^2},\qquad x\neq 0,
    	\end{equation}
    	and set $\nu(0):= \operatorname{tr}A(0)/n = 1$ since $A(0)=I_{n\times n}$.
    		By \eqref{lips-1}, $\lambda\le \nu(x)\le \Lambda$ for all $x\neq 0$.
    		Moreover, from \eqref{eq:Aclose} and $A(0)=I_{n\times n}$,
    		\begin{equation}\label{eq:mucmp1}
    			|\nu(x)-1|\le C L_A r
    		\end{equation}
    		 for $|x|\le r$ with $C=C(n,\lambda,\Lambda)$.
    	%\end{remark}
    	
    	Fix a parameter $\alpha\ge 2$.
    	As in section 1, we define 
    \begin{equation}\label{eq:Hdef}
    		H(r):=\int_{B_r} u(x)^2\,\nu(x)\,e^{\alpha-1}_r(x)\,dx,
    	\end{equation}
        with $e_r=r^2-|x|^2.$
    	Denote the Dirichlet and potential terms as
    \begin{equation}\label{eq:DLEdef}
    		D(r):=\int_{B_r} a_{ij}(x)\,\partial_i u\,\partial_j u\,e_r^\alpha\,dx,\quad
    		L(r):=\int_{B_r} V_\varepsilon(x)\,u^2\,e_r^\alpha\,dx,\quad
    		E(r):=\int_{B_r} f_\varepsilon(x)\,u^2\,e_r^\alpha\,dx,
    	\end{equation}
    	and
    	\begin{equation}\label{eq:I_N_def}
    		I(r):=D(r)+L(r),\qquad \mathcal{N}(r):=\frac{I(r)}{H(r)}.
    	\end{equation}
    We also introduce 
 	\begin{equation}\label{eq:I0def}
    		I_0(r):=-\int_{B_r}u\,a_{ij}\partial_i u\,\partial_j\!\bigl(e_r^\alpha\bigr)\,dx.
    	\end{equation}
    	Since $\partial_j(e_r^\alpha)=-2\alpha x_j e_r^{\alpha-1}$,
    	we have
    	\begin{equation}\label{eq:I0rewrite}
    		I_0(r)=2\alpha\int_{B_r} u\,\bigl(a_{ij}\partial_i u\,x_j\bigr)\,e_r^{\alpha-1}\,dx.
    	\end{equation}

 %  {First identity: $I_0=I-E$ }
    	
    	%\begin{lemma}\label{lem:I0}
       % [Identity $I_0=I-E$]
    	Using the integration by parts arguments as in Lemma \ref{lem:I0 Simple}, we can show the following identity,	
    \begin{equation}\label{eq:I0IE}
    			I_0(r)=D(r)+L(r)-E(r)=I(r)-E(r).
    		\end{equation}
         
    	%\end{lemma}
The Lemma \ref{lem:hprime} and Lemma \ref{lem:Dprime_variable} below are standard and can be obtained by integration by parts. We include the proof in the Appendix for the completion of the presentation.
    \begin{lemma}\label{lem:hprime}
    	One has the identity
    		\begin{equation}\label{eq:Hprime_final_413}
    		H'(r)=\frac{2\alpha+n-2}{r}\,H(r)+\frac{I-E}{\alpha r}+\frac{1}{r}\,\mathcal{E}_H(r),
    		\end{equation}
    	where
    		\begin{equation}\label{eq:EH_def_412}	\mathcal{E}_H(r):=\int_{B_r}|u|^2\Big[\operatorname{div}(Ax)\,\nu^{-1}-n\Big]\nu\,e_r^{\alpha-1}\,dx.
    		\end{equation}
    		Moreover,
\begin{equation}\label{eq:EH_bound_414}
    			|\mathcal{E}_H(r)|\le c\,L_A\,r\,H(r).
    		\end{equation}
    	
    	\end{lemma}

    Then
         Lemma \ref{lem:hprime} gives that 
\begin{equation}\label{eq:Hprime_key}
    		\left|\,H'(r)-\frac{2\alpha+n-2}{r}H(r)-\frac{1}{\alpha r}\big(I(r)-E(r)\big)\right|
    		\le C L_A H(r),
    	\end{equation}
    	with $C=C(n,\lambda,\Lambda)$.
    	
    We now turn to the derivative of $D(r)$. 
    Define, for $x\neq 0$,
    \begin{equation}\label{eq:Zdef}
    	\omega(x):=\frac{A(x)x}{\nu(x)} ,\qquad\text{then}\qquad \omega(x)\cdot x=\frac{A(x)x\cdot x}{\nu(x)}=|x|^2.
    \end{equation}
   % (Here $(A(x)x)_j=a_{ij}(x)x_i$.) Moreover, using $A(0)=I$, \eqref{eq:Aclose}, and $|\nu(x)-1|\le C L_A r$ for $|x|\le r$,
    The following properties for $\omega(x)$ hold
\begin{equation}\label{eq:Zbounds}
    \begin{aligned}
    |\omega(x)-x|\le C L_A |x|^2
       ,\qquad
    	|\nabla \omega(x)-I|\le C L_A |x|,\qquad
    	| \mathrm{div} \omega(x)-n|\le C L_A |x|,
        \end{aligned}
    \end{equation}
   and it holds that
    \begin{align}\label{AAA-1}
    |{\rm div}(Ax)\nu^{-1}-n|\le CL_A|x|,
    \end{align}
     where $C$ depends on $n, \lambda, \Lambda$.  See e.g. \cite{D25}, \cite{GPS18}, \cite{ZZ23}.
     \begin{lemma}\label{lem:Dprime_variable}
    	There exists $C=C(n,\lambda,\Lambda)>0$ such that for every $0<r\le R$,
    	\begin{equation}\label{eq:Dprime_formula}
    		\begin{aligned}
    			D'(r)=&\ \frac{2\alpha+n-2}{r}D(r)
    			+\frac{4\alpha}{r}\int_{B_r}(A\nabla u\cdot x)^2\,\nu^{-1}\,e_r^{\alpha-1}\,dx
    				 \\
    			&-\frac{2}{r}\int_{B_r}(A\nabla u\cdot x)\,\nu^{-1}\,{\rm div}(A\nabla u)\,e_r^\alpha\,dx
    			+\frac{1}{r}\mathcal{E}_D(r),
    		\end{aligned}
    	\end{equation}
    		where the error term is
\begin{equation}\label{eq:ED_definition}
    		\begin{aligned}
    			\mathcal{E}_D(r):=&\ \int_{B_r}({\rm div} \omega-n)\,a_{ij}\partial_i u\,\partial_j u\,e_r^\alpha\,dx
    			+2\int_{B_r}a_{ij}\partial_k u\,\partial_j u\,(\delta_{ik}-\partial_i \omega_k)\,e_r^\alpha\,dx \\
    			&\ +\int_{B_r}(\partial_k a_{ij})\partial_i u\,\partial_j u\,\omega_k\,e_r^\alpha\,dx.
    		\end{aligned}
    	\end{equation}
    	Moreover, the error term satisfies
    	\begin{equation}\label{eq:ED_bound}
    		|\mathcal{E}_D(r)|\le C L_A r\,D(r).
    	\end{equation}
    \end{lemma}
	
    We need to  consider the derivative of $I$. Recall 
        $I(r)=D(r)+L(r)$ 
         and $L=\int_{B_r}V_\varepsilon u^2e_r^\alpha$.
    	Applying Lemma \ref{lem:diffF} to $G(x)=V_\varepsilon(x)u^2(x)$gives that
\begin{equation} \label{eq:Lprime_formula}
    \begin{aligned}
    		L'(r)&=\frac{2\alpha+n}{r}\int_{B_r}V_\varepsilon u^2e_r^\alpha+\frac{1}{r}\int_{B_r}x\cdot\nabla(V_\varepsilon u^2)\,e_r^\alpha \\
            &=
            \frac{2\alpha+n}{r}L(r)+\frac{1}{r}\int_{B_r}x\cdot\nabla  V_\varepsilon\,u^2\,e_r^\alpha\,dx 
    		+\frac{2}{r}\int_{B_r}V_\varepsilon\,u\,(x\cdot\nabla u)\,e_r^\alpha\,dx .
    \end{aligned}
\end{equation}

    Adding \eqref{eq:Dprime_formula} and \eqref{eq:Lprime_formula}, and using (\ref{eq:PDEeps}), we conclude that
\begin{equation}\label{eq:Iprime_formula}
    	\begin{aligned}
    		I'(r)=&\ \frac{2\alpha+n-2}{r}I(r)+\frac{2}{r}L(r)
    		+\frac{4\alpha}{r}\int_{B_r}\Big(A\nabla u\cdot x\Big)^2\nu^{-1}\,e_r^{\alpha-1}\,dx \\
    		&\ +\frac{2}{r}\int_{B_r}f_\varepsilon\,u\,\Big(
            A\nabla u\cdot x
            \Big)\nu^{-1}\,e_r^{\alpha}\,dx
    		+\frac{1}{r}\int_{B_r}x\cdot\nabla V_\varepsilon\,u^2\,e_r^\alpha\,dx \\
    		&\ + B_V(r) + \frac{1}{r}\mathcal{E}_D(r),
    	\end{aligned}
    \end{equation}
    where
    \begin{equation}\label{eq:BVdef}
    	B_V(r):=\frac{2}{r}\int_{B_r}V_\varepsilon\,u\,\Big((x\cdot\nabla u)-\nu^{-1}\big(A\nabla u\cdot x)\Big)\,e_r^\alpha\,dx .
    \end{equation}
 Next  we show an upper bound for  $B_V$.
    \begin{lemma}\label{lem:BV_bound}
    	There exists $C=C(n,\lambda,\Lambda)>0$ such that for every $0<r< R$,
    	\begin{equation}\label{eq:BV_bound}
    		|B_V(r)|\le C L_A\Big(Mr^2+M_0\varepsilon^{\beta-1}r^3\Big)\,H(r).
    	\end{equation}
    \end{lemma}
    \begin{proof}
    	Recall $\omega(x)$ from \eqref{eq:Zdef}.  Note that $\nu^{-1}(A\nabla u\cdot x)=\omega\cdot\nabla u$.
    	Hence
   \begin{align*}
    	B_V(r)&=\frac{2}{r}\int_{B_r}V_\varepsilon\,u\,\big((x-\omega)\cdot\nabla u\big)\,e_r^\alpha\,dx
    	=-\frac{1}{r}\int_{B_r}u^2\, {\rm div}\!\big(V_\varepsilon(x)(x-\omega)e_r^\alpha\big)\,dx.
    	\end{align*}
    	Expanding the divergence to have
    	\[
    	{\rm div}\!\big(V_\varepsilon(x)(x-\omega)e_r^\alpha\big)
    	=(x-\omega)\cdot\nabla V_\varepsilon(x)\,e_r^\alpha
    	+V_\varepsilon(x)\,{\rm div}(x-\omega)\,e_r^\alpha
    	+V_\varepsilon(x)(x-\omega)\cdot\nabla(e_r^\alpha).
    	\]
    	The last term vanishes because $\nabla(e_r^\alpha)=-2\alpha x\,e_r^{\alpha-1}$ and
    	$(x-\omega)\cdot x=|x|^2-\omega\cdot x=0$ by \eqref{eq:Zdef}. Thus,
    	\[
    	|B_V(r)|\le \frac{1}{r}\int_{B_r}u^2\Big(|x-\omega|\,|\nabla V_\varepsilon|+|V_\varepsilon|\,|{\rm div} \omega-n|\Big)e_r^\alpha\,dx.
    	\]
    	Thanks to  \eqref{eq:Zbounds} and \eqref{VVV-1}, we obtain
    	\begin{align*}
    	    |B_V(r)|\le C L_A\Big(M_0\varepsilon^{\beta-1}r+M\Big)\int_{B_r}u^2e_r^\alpha\,dx.
    	\end{align*}
    	Using the fact that
    	 $e_r^\alpha\le r^2e_r^{\alpha-1}$,
         %and $\nu\simeq 1$ on $B_{r}$, 
         we derive \eqref{eq:BV_bound}.
    \end{proof}
    
    With the derivative of $H, D$ and $I$, we are able to show the almost monotonicity of the frequency function.
\begin{lemma}\label{lem:monotoneN}
    	There exists $C=C(n,\lambda,\Lambda)>0$ such that 
\begin{align*}
 \widetilde {\mathcal{N}}:= e^{CL_{A}r}\big(\mathcal{N}(r) +P(r) \big) 
\end{align*}
        is monotone non-decreasing for every $0<r< R$, where $P(r)=M r^2+CM_0 \varepsilon^{\beta-1} r^3+ M_0^2\varepsilon^{2\beta} r^4+CL_{A}M r^3+CL_{A} M_0 r^4.$
      %  $\widetilde N$]
    %	For every $0<r\leq R$ one has $\widetilde N'(r)\ge 0$. In particular $\widetilde N$ is nondecreasing on $(0,R]$.
    	%\begin{equation}\label{eq:Nprime_lower_variable}
    %		\mathcal{N}'(r)+C L_A \mathcal{N}(r)\ \ge\ -2Mr - C K r^2\varepsilon^{\beta-1}-\frac{C}{4\alpha}K^2 r^3\varepsilon^{2\beta}
    	%	\;-\; C L_A Mr^2 \;-\; C L_A K\varepsilon^{\beta-1}r^3 .
    %	\end{equation}
    \end{lemma}
    
    \begin{proof}
 %\begin{equation}\label{eq:EHdef}
    	%	H'(r)=\frac{2\alpha+n-2}{r}H(r)+\frac{1}{\alpha r}\big(I(r)-E(r)\big)+\frac{1}{r}E_H(r),
    	%	\qquad |E_H(r)|\le C L_A r\,H(r).
    	%\end{equation}
    	Thanks to the fact that $H^2\mathcal{N}'=I'H-H'I$, \eqref{eq:Iprime_formula}, and \eqref{eq:Hprime_final_413}, we obtain that
\begin{equation}\label{eq:H2Nprime_step1}
    		\begin{aligned}
    			H(r)^2\mathcal{N}'(r)=&
    			\ \frac{2}{r}H(r)L(r)
    			+H(r)\Bigg(
    			\frac{4\alpha}{r}\!\int_{B_r}\!\Big(A\nabla u\cdot x\Big)^2\nu^{-1}e_r^{\alpha-1}\,dx \\
    			&
    			+\frac{2}{r}\!\int_{B_r}\!f_\varepsilon u\Big(A\nabla u\cdot x\Big)\nu^{-1}e_r^\alpha\,dx
    			+\frac{1}{r}\!\int_{B_r}\!x\cdot\nabla V_\varepsilon\,u^2e_r^\alpha\,dx
    			\Bigg) \\
    			&\ +H(r)B_V(r)+\frac{1}{r}H(r)\mathcal{E}_D(r)
    			-\frac{1}{\alpha r}\big(I(r)-E(r)\big)I(r)-\frac{1}{r}I(r)\mathcal{E}_H(r).
    		\end{aligned}
    	\end{equation}
    	
    	Now set
    	$\Phi:=\Big(A\nabla u\cdot x\Big)+\frac{e_r}{4\alpha}f_\varepsilon u.$
    	Then
    	\[
    	\Phi^2\nu^{-1}e_r^{\alpha-1}
    	=\Big(A\nabla u\cdot x\Big)^2\nu^{-1}e_r^{\alpha-1}
    	+\frac{1}{2\alpha}\,f_\varepsilon u\Big(A\nabla u\cdot x\Big)\nu^{-1}e_r^\alpha
    	+\frac{1}{16\alpha^2}f_\varepsilon^2u^2\nu^{-1}e_r^{\alpha+1}.
    	\]
    	Therefore,
    	\begin{equation}\label{eq:Phi_identity_variable}
    		\begin{aligned}
    			\int_{B_r} f_\varepsilon u\Big(A\nabla u\cdot x\Big)\nu^{-1}e_r^\alpha\,dx
    			&=
    			2\alpha\!\int_{B_r}\!\Phi^2\nu^{-1}e_r^{\alpha-1}\,dx
-2\alpha\!\int_{B_r}\!\Big(A\nabla u\cdot x\Big)^2\nu^{-1}e_r^{\alpha-1}\,dx\\
    			&-\frac{1}{8\alpha}\!\int_{B_r}\!f_\varepsilon^2u^2\nu^{-1}e_r^{\alpha+1}\,dx.
    		\end{aligned}
    	\end{equation}
Substituting \eqref{eq:Phi_identity_variable} into \eqref{eq:H2Nprime_step1} to get
    	\begin{equation}\label{eq:H2Nprime_step2}
    		\begin{aligned}
    			H^2\mathcal{N}'=&\ \frac{2}{r}HL+\frac{4\alpha H}{r}\int_{B_r}\Phi^2\nu^{-1}e_r^{\alpha-1}\,dx
    			-\frac{H}{4\alpha r}\int_{B_r}f_\varepsilon^2u^2\nu^{-1}e_r^{\alpha+1}\,dx \\
    			&\ +\frac{H}{r}\int_{B_r}x\cdot\nabla V_\varepsilon\,u^2e_r^\alpha\,dx
    			+HB_V+\frac{1}{r}H\mathcal{E}_D-\frac{1}{\alpha r}(I-E)I-\frac{1}{r}I \mathcal{E}_H .
    		\end{aligned}
    	\end{equation}
    	Set
    	\[
J(r):=\int_{B_r}u\,\Phi\,e_r^{\alpha-1}\,dx
    	=\int_{B_r}u\Big(A\nabla u\cdot x\big)e_r^{\alpha-1}\,dx+\frac{1}{4\alpha}\int_{B_r}f_\varepsilon u^2e_r^\alpha\,dx.
    	\]
    	Note that $I_0=I-E$. From \eqref{eq:I0rewrite} and (\ref{eq:I0IE}), we get
\begin{equation}\label{eq:J_identity_variable}
    		J(r)=\frac{I(r)-E(r)}{2\alpha}+\frac{E(r)}{4\alpha}=\frac{2I(r)-E(r)}{4\alpha}.
    	\end{equation}
    	Consequently,
\begin{equation}\label{eq:IE_identity_variable}
    		-\frac{1}{\alpha r}(I-E)I=-\frac{4\alpha}{r}J^2(r)+\frac{E^2(r)}{4\alpha r}.
    	\end{equation}
    	Inserting \eqref{eq:IE_identity_variable} into \eqref{eq:H2Nprime_step2} to obtain\begin{equation}\label{eq:H2Nprime_step3}
    		\begin{aligned}
    			H^2\mathcal{N}'=&\ \frac{2}{r}HL+\frac{4\alpha}{r}\Big(H\int_{B_r}\Phi^2\nu^{-1}e_r^{\alpha-1}\,dx-J^2(r)\Big) \\
    			&\ -\frac{H}{4\alpha r}\int_{B_r}f_\varepsilon^2u^2\nu^{-1}e_r^{\alpha+1}\,dx+\frac{E(r)^2}{4\alpha r}
    			+\frac{1}{r}H\int_{B_r}x\cdot\nabla V_\varepsilon\,u^2e_r^\alpha\,dx \\
    			&\ +HB_V+\frac{1}{r}H\mathcal{E}_D-\frac{1}{r}I \mathcal{E}_H .
    		\end{aligned}
    	\end{equation}
    	By Cauchy--Schwarz with the weight function $\nu$, it holds that
    	\[
    	J^2(r)=\Big(\int u\Phi\,e_r^{\alpha-1}\Big)^2
    	\le \Big(\int u^2\nu\,e_r^{\alpha-1}\Big)\Big(\int \Phi^2\nu^{-1}e_r^{\alpha-1}\Big)
    	=H\int \Phi^2\nu^{-1}e_r^{\alpha-1}.
    	\]
    	Hence the second term on the right hand side of  \eqref{eq:H2Nprime_step3} is nonnegative. 
    	Hence \eqref{eq:H2Nprime_step3} leads to
\begin{equation}\label{eq:H2Nprime_lower_basic}
    		\begin{aligned}
    			H^2\mathcal{N}'\ \ge\ &
    			-\frac{2}{r}H|L|
    			-\frac{H}{4\alpha r}\int_{B_r}f_\varepsilon^2u^2\nu^{-1}e_r^{\alpha+1}\,dx
    			+\frac{1}{r}H\int_{B_r}x\cdot\nabla V_\varepsilon\,u^2e_r^\alpha\,dx \\
    			&\ -\frac{1}{r}H|\mathcal{E}_D|
    			-\frac{1}{r}|I||\mathcal{E}_H|
    			+H |B_V| .
    		\end{aligned}
    	\end{equation}
    	
    	We now estimate each term on the right-hand side (\ref{eq:H2Nprime_lower_basic}).
    	Using $\|V_\varepsilon\|_\infty\le M$ and  $\lambda \leq \nu\leq \Lambda$,
    	\begin{equation}\label{LLL-1}
    	|L(r)|\le M\int_{B_r}u^2e_r^\alpha\,dx \le CMr^2\int_{B_r}u^2\nu e_r^{\alpha-1}\,dx = CMr^2H(r).
    	\end{equation}
    	Then $-\frac{2}{r}\frac{|L|}{H}\ge -C Mr$.
   It follows from (\ref{VVV-1}) that
    	\[
    	\frac{1}{rH}\left|\int_{B_r}x\cdot\nabla V_\varepsilon\,u^2e_r^\alpha\,dx\right|
    	\le \frac{C\|\nabla V_\varepsilon\|_\infty}{H}\int_{B_r}u^2e_r^\alpha\,dx
    	\le C M_0 r^2\varepsilon^{\beta-1}.
    	\]
    	By $e_r^{\alpha+1}\le r^4e_r^{\alpha-1}$, we have
    	\[
    	\frac{1}{4\alpha rH}\int_{B_r}f_\varepsilon^2u^2\nu^{-1}e_r^{\alpha+1}\,dx
    	\le \frac{C\|f_\varepsilon\|_\infty^2}{4\alpha rH}\,r^4\int_{B_r}u^2e_r^{\alpha-1}\,dx
    	\le \frac{CM_0^2 r^3\varepsilon^{2\beta}}{4\alpha}.
    	\]
    	Note that $I(r)=D(r)+L(r)$. For the error term $\mathcal{E}_D$, \eqref{eq:ED_bound} gives
    	\[
    	\frac{|\mathcal{E}_D(r)|}{rH(r)}\le C L_A\frac{D(r)}{H(r)}\le C L_A\Big(|\mathcal{N}(r)|+\frac{|L(r)|}{H(r)}\Big)\le C L_A\big(|\mathcal{N}(r)|+Mr^2\big).
    	\]
    	The estimate \eqref{eq:EH_bound_414} in Lemma \ref{lem:hprime} gives that $\frac{|\mathcal{E}_H(r)|}{rH(r)}\le C L_A$. Hence
    	\[
    	\frac{|I(r)|}{H(r)}\frac{|\mathcal{E}_H(r)|}{rH(r)}\le C L_A |\mathcal{N}(r)|.
    	\]
    	Furthermore, Lemma \ref{lem:BV_bound} yields that
    	\[
    	\frac{|B_V(r)|}{H(r)}\le C L_A\big(Mr^2+M_0\varepsilon^{\beta-1}r^3\big).
    	\]
    	Dividing \eqref{eq:H2Nprime_lower_basic} by $H^2(r)$ and combining  the last four inequalities gives that
\begin{equation}\label{eq:Nprime_lower_abs}
    		\mathcal{N}'(r)\ \ge\ -2Mr - C M_0 r^2\varepsilon^{\beta-1}-\frac{C}{4\alpha}M_0^2 \varepsilon^{2\beta}r^3
    		- C L_A|\mathcal{N}(r)|-C L_A Mr^2 - C L_A M_0\varepsilon^{\beta-1}r^3 .
    	\end{equation}
    	
    	It remains to remove the absolute value on $\mathcal{N}(r)$ in the right side of the last inequality. Since $D(r)\ge 0$ and $\mathcal{N}(r)=\frac{D(r)+L(r)}{H(r)}$, we have
    	$\mathcal{N}(r)\ge -\frac{|L(r)|}{H(r)}\ge -C Mr^2$, where we used  (\ref{LLL-1}). Consequently $|\mathcal{N}(r)|\le \mathcal{N}(r)+C Mr^2$.
    	Substituting this into \eqref{eq:Nprime_lower_abs} to imply
    	that
    \begin{equation}\label{eq:Nprime_almost}
    	\mathcal{N}'(r)+c_0L_A \mathcal{N}(r)\ \ge\ -2Mr - C M_0r^2\varepsilon^{\beta-1}-\frac{C}{4\alpha}M_0^2 r^3\varepsilon^{2\beta}
    	\;-\; C L_A Mr^2 \;-\; C L_A M_0\varepsilon^{\beta-1}r^3 .
    \end{equation}

    Set
    \[
    a_1:=2M,\quad a_2:=CM_0\varepsilon^{\beta-1},\quad a_3:=\frac{C}{4\alpha}M_0^2\varepsilon^{2\beta},
    \quad b_2:=CL_A M,\quad b_3:=CL_A M_0\varepsilon^{\beta-1},
    \]
    and define
    \begin{equation}\label{eq:Pdef}
    	P(r):=\frac{a_1}{2}r^2+\frac{a_2}{3}r^3+\frac{a_3}{4}r^4+\frac{b_2}{3}r^3+\frac{b_3}{4}r^4,
    	\qquad 
    	\widetilde {\mathcal{N}}(r):=e^{c_0L_A r}\big(\mathcal{N}(r)+P(r)\big).
    \end{equation}
    	Differentiating $\widetilde {\mathcal{N}}$ in \eqref{eq:Pdef} with respect to $r$ yields that
    	\[
    	\widetilde {\mathcal{N}}'(r)=e^{c_0L_A r}\Big(\mathcal{N}'(r)+c_0L_A \mathcal{N}(r)+P'(r)+c_0L_A P(r)\Big).
    	\]
    	Since $c_0L_A P(r)\ge 0$, thanks to  \eqref{eq:Nprime_almost}, we obtain
    	\[
    	\widetilde {\mathcal{N}}'(r)\ \ge\ c_0e^{c_0L_A r}L_A P(r)\ge 0 .
    	\]
    This finishes the proof of the Lemma.
    \end{proof}

With aid of the monotonicity of $\widetilde {\mathcal{N}}$, we are ready to prove Theorem
\ref{thm:threeballs-II-opt-correct}.
   
\begin{proof}[Proof of Theorem
\ref{thm:threeballs-II-opt-correct} ]   By \eqref{eq:Hprime_key}, we get the bound on $E/H$.
    Indeed, since $|f_\varepsilon|\le C M_0\varepsilon^\beta$, then
    \begin{equation}\label{eq:EoverH_bound}
    	\left|\frac{E(r)}{\alpha r H(r)}\right|
    	\le \frac{1}{\alpha r H(r)}\int_{B_r}|f_\varepsilon|u^2e_r^\alpha
    	\le \frac{CM_0\varepsilon^\beta}{\alpha r H(r)}\,r^2\int_{B_r}u^2e_r^{\alpha-1}
    	\le \frac{CM_0\varepsilon^\beta}{\alpha}\,r.
    \end{equation}
    Dividing \eqref{eq:Hprime_key} by $H(r)$ and using \eqref{eq:EoverH_bound} yields
    \begin{equation}
    \begin{aligned}\label{eq:Hlog_ineq}
    	\frac{H'(r)}{H(r)}
    	\ge \frac{2\alpha+n-2}{r}+\frac{1}{\alpha r}\mathcal{N}(r)-CL_A-\frac{CM_0\varepsilon^\beta}{\alpha}r \\
    	\frac{H'(r)}{H(r)}
    	\le \frac{2\alpha+n-2}{r}+\frac{1}{\alpha r}\mathcal{N}(r)+CL_A+\frac{CM_0\varepsilon^\beta}{\alpha}r .
        \end{aligned}
    \end{equation}
    Utilizing \eqref{eq:Pdef}, we may rewrite
    \begin{equation}\label{eq:Nrewrite}
    	\mathcal{N}(r)=e^{-c_0L_A r}\widetilde {\mathcal{N}}(r)-P(r).
    \end{equation}
    Substituting \eqref{eq:Nrewrite} into \eqref{eq:Hlog_ineq} gives
    \begin{equation}\label{eq:Hlog_ineq_tildeN}
    	\frac{H'(r)}{H(r)}
    	\ge \frac{2\alpha+n-2}{r}+\frac{e^{-c_0L_A r}}{\alpha r}\widetilde {\mathcal{N}}(r)-\frac{P(r)}{\alpha r}-CL_A-\frac{CM_0\varepsilon^\beta}{\alpha}r,
    \end{equation}
    \begin{equation}\label{eq:Hlog_ineq_tildeN_upper}
    	\frac{H'(r)}{H(r)}
    	\le \frac{2\alpha+n-2}{r}+\frac{e^{-c_0L_A r}}{\alpha r}\widetilde{\mathcal{N}}(r)-\frac{P(r)}{\alpha r}+CL_A+\frac{CM_0\varepsilon^\beta}{\alpha}r.
    \end{equation}
    
   Integrating \eqref{eq:Hlog_ineq_tildeN} from $2r_2$ to $r_3$, using the monotonicity of  $\widetilde {\mathcal{N}}$, 
   we obtain that
    \begin{align}
    	\log\frac{H(r_3)}{H(2r_2)}
    	&=\int_{2r_2}^{r_3}\frac{H'(r)}{H(r)}\,dr \notag\\
    	&\ge \int_{2r_2}^{r_3}\Bigg[\frac{2\alpha+n-2}{r}+\frac{e^{-c_0L_A r_3}}{\alpha r}\widetilde{\mathcal{N}}(2r_2)
    	-\frac{P(r)}{\alpha r}-CL_A-\frac{CM_0\varepsilon^\beta}{\alpha}r\Bigg]dr 
    	\label{eq:int_lower}.
    \end{align}
    we can further bound \begin{equation}\label{eq:int_lower_simplified}
    	\log\frac{H(r_3)}{H(2r_2)}
    	\ge \Bigg[2\alpha+n-2+\frac{e^{-c_0L_A r_3}}{\alpha}\widetilde{\mathcal{N}}(2r_2)\Bigg]\log\frac{r_3}{2r_2}
    	-\frac{1}{\alpha}\int_{2r_2}^{r_3}\frac{P(r)}{r}\,dr
    	-CL_A r_3-\frac{CM_0\varepsilon^\beta}{2\alpha}r_3^2.
    \end{equation}
    Evaluating the integral of $P(r)/r$ using \eqref{eq:Pdef} yields
    \begin{equation}\label{eq:Pr_int_bound}
    	\int_{2r_2}^{r_3}\frac{P(r)}{r}\,dr
    	\le \frac{a_1}{4}r_3^2+\frac{a_2}{9}r_3^3+\frac{a_3}{16}r_3^4+\frac{b_2}{9}r_3^3+\frac{b_3}{16}r_3^4 .
    \end{equation}  
Integrating 
  \eqref{eq:Hlog_ineq_tildeN_upper} from $r_1$ to $2r_2$, by monotonicity of $\widetilde {\mathcal{N}}$ and the fact $-P(r)/(\alpha r)\le 0$, we obtain
    \begin{align}
    	\log\frac{H(2r_2)}{H(r_1)}
    	&=\int_{r_1}^{2r_2}\frac{H'(r)}{H(r)}\,dr \notag\\
    	&\le \int_{r_1}^{2r_2}\Bigg[\frac{2\alpha+n-2}{r}+\frac{e^{-c_0L_A r_1}}{\alpha r}\widetilde N(2r_2)
    	+CL_A+\frac{CM_0\varepsilon^\beta}{\alpha}r\Bigg]dr \notag\\
    	&\le \Bigg[2\alpha+n-2+\frac{e^{-c_0L_A r_1}}{\alpha}\widetilde {\mathcal{N}}(2r_2)\Bigg]\log\frac{2r_2}{r_1}
+CL_A(2r_2)+\frac{CM_0\varepsilon^\beta}{2\alpha}(2r_2)^2 .
    	\label{eq:int_upper}
    \end{align}
    Introduce the notations, \begin{equation}\label{eq:beta_gamma_def}
    	\xi:=e^{-c_0L_A r_1}\log\frac{2r_2}{r_1},
    	\qquad
    	\eta:=e^{-c_0L_A r_3}\log\frac{r_3}{2r_2}.
    \end{equation}
  By \eqref{eq:int_lower_simplified}--\eqref{eq:int_upper}, the following three-ball  estimate hold for $H$
    \begin{equation}\label{eq:threeball_H}
    	H(2r_2)\ \le\ H^{\frac{\eta}{\xi+\eta}}(r_1)\,H^{\frac{\xi}{\xi+\eta}}\,(r_3)
    	\exp\Bigg\{C L_A r_3+\frac{C M_0\varepsilon^\beta}{\alpha}r_3^2
    	+\frac{C}{\alpha}\Big(a_1 r_3^2+a_2 r_3^3+a_3 r_3^4+b_2 r_3^3+b_3 r_3^4\Big)\Bigg\},
    \end{equation}
    where $C=C(n,\lambda,\Lambda)>0$ and $a_1,a_2,a_3,b_2,b_3$ are in (\ref{eq:Pdef}).

    We also want to get rid of the weight function in $e_r$ in $H(r).$
   Let $ h(r):=\int_{B_r}u(x)^2\,dx.$
    Since  $e_r^{\alpha-1}\le r^{2(\alpha-1)}$, we have  \begin{equation}\label{eq:96_rederive}
    	\begin{aligned}
    		H(r)
    =\int_{B_r}u^2\,\nu\,e_r^{\alpha-1}\,dx
    		\le C\,r^{2(\alpha-1)}\,h(r),
    	\end{aligned}
    \end{equation}

    For $0<r<\rho\leq R$.
    We obtain
\begin{equation}\label{eq:97_rederive}
    	\begin{aligned}
    		H(\rho)
    		=\int_{B_\rho}u^2\,\nu\,e_\rho^{\alpha-1}\,dx
    		\geq c\,(\rho^2-r^2)^{\alpha-1}\,h(r),
    	\end{aligned}
    \end{equation}
    
    Thus, \begin{equation}\label{eq:hr2_from_H2r2}
    	h(r_2)\ \le\ \frac{1}{c}\left(\frac{4}{3}\right)^{\alpha-1}(2r_2)^{-2(\alpha-1)}\,H(2r_2).
    \end{equation}
  
    %Applying \eqref{eq:96_rederive} with $r=r_1$ and $r=r_3$ yields
    %\begin{equation}\label{eq:Hr1_upper}
    %	H(r_1)\le C\,r_1^{2(\alpha-1)}h(r_1),\qquad
    %	H(r_3)\le C\,r_3^{2(\alpha-1)}h(r_3).
  %  \end{equation}
    	By (\ref{eq:96_rederive}) and (\ref{eq:hr2_from_H2r2}), \eqref{eq:threeball_H} becomes
    		\begin{align}
    		h(r_2)
    		&\leq C'\left(\frac{4}{3}\right)^{\alpha-1}
    		\left(\frac{r_1}{2r_2}\right)^{\frac{2(\alpha-1)\eta}{\xi+\eta}}
    		\left(\frac{r_3}{2r_2}\right)^{\frac{2(\alpha-1)\xi}{\xi+\eta}}
    		\,h^{\frac{\eta}{\xi+\eta}}(r_1)\,h^{\frac{\xi}{\xi+\eta}}(r_3)
    		\exp\Big\{\mathcal{F}(\alpha,\varepsilon;r_3)\Big\},
    		\label{eq:hr2_sub2}
    	\end{align}
    	where 
    \begin{equation}\label{eq:F_def}
    		\mathcal{F}(\alpha,\varepsilon;r_3)
    		:=CL_A r_3+\frac{CM_0\varepsilon^\beta}{\alpha}\,r_3^2
    		+\frac{C}{\alpha}\big(a_1r_3^2+a_2r_3^3+a_3r_3^4+b_2r_3^3+b_3r_3^4\big).
    	\end{equation}
       % \textcolor{red}{ we need to calculate this out.
    	%	$\left(\frac{r_1}{2r_2}\right)^{\frac{2(\alpha-1)\eta}{\xi+\eta}}
    	%	\left(\frac{r_3}{2r_2}\right)^{\frac{2(\alpha-1)\xi}{\xi+\eta}}
    	%	\, $. 
        %It seems that it does not match the three-ball ineqality below}
    	where $C'=C'(n,\lambda,\Lambda)>0$.
    	Then we can write \eqref{eq:hr2_sub2} in the form
\begin{equation}\label{eq:98_rederive_general}
    		h(r_2)\ \le\ h^{\frac{\eta}{\xi+\eta}}(r_1)\,h^{\frac{\xi}{\xi+\eta}}(r_3)\,
    		\exp\{
    		C(\alpha-1)+\frac{2(\alpha-1)\eta}{\xi+\eta} \log \frac{r_1}{2r_2}+\frac{2(\alpha-1)\xi}{\xi+\eta} \log\frac{r_3}{2r_2}+\mathcal{F}(\alpha,\varepsilon; r_3)\big)
    		\}.
    	\end{equation}

    	%Recalling the definitions 
    	%\[
    	%a_1:=M,\qquad a_2:=K\varepsilon^{\beta-1},\qquad a_3:=\frac{K^2\varepsilon^{2\beta}}{\alpha},
    	%\qquad b_2:=L_A M,\qquad b_3:=L_AK\varepsilon^{\beta-1},
    	%\]
    	 % the bracket of $\mathcal{F}(\alpha,\varepsilon;r_3)$ in \eqref{eq:98_rederive_general} becomes
    	%\[
    	%\frac{C}{\alpha}\big(a_1r_3^2+a_2r_3^3+a_3r_3^4+b_2r_3^3+b_3r_3^4\big)
    	%= \frac{C}{\alpha}\left(Mr_3^2+K\varepsilon^{\beta-1}r_3^3+\frac{K^2\varepsilon^{2\beta}}{\alpha}r_3^4\right)
    	%+\frac{C}{\alpha}\left(L_A Mr_3^3+L_AK\varepsilon^{\beta-1}r_3^4\right).
    	%\]
   Note that $\log\frac{r_1}{2r_2}<0$. Recall $a_1, a_2, a_3$ and $b_1, b_2$ in (\ref{eq:Pdef}).
  The exponent in the (\ref{eq:98_rederive_general}) is bounded by
    \begin{equation}\label{eq:E_explicit}
    		\begin{aligned}
    			  C(\alpha-1)[\log\frac{r_3}{2r_2}+1] &+CL_A R
    			+\frac{CM_0\varepsilon^\beta}{\alpha}R^2
    			+\frac{CM}{\alpha}R^2
    			+\frac{CM_0}{\alpha}\varepsilon^{\beta-1}R^3 \\
    			&+\frac{CM_0^2}{\alpha^2}\varepsilon^{2\beta}R^4
    			 +\frac{CL_A M}{\alpha}R^3
    			+\frac{CL_A M_0}{\alpha}\varepsilon^{\beta-1}R^4 .
    	\end{aligned}
    \end{equation}
    Finally we optimize \eqref{eq:E_explicit} by determining $\varepsilon$ and $\alpha.$
     We may assume $M_0\varepsilon^\beta>1$, Otherwise, the $L^\infty$ norm of $V-V_{\varepsilon}$ is bounded in (\ref{VVV-1}). We also choose $M=\max\{ M, \ M_0\}$ since $\|V\|_{C^{0,\beta}}\leq M$.
     Assume $R>1$ and $L_A\geq 1$ for some fixed $R$ and $L_A$.
Thus, we have $ \frac{CM}{\alpha}R^2\leq  \frac{CL_A M}{\alpha}R^3$, $\frac{CL_A M}{\alpha}R^3\leq \frac{CL_{A} M\varepsilon^{\beta-1}}{\alpha}R^4,$  and 
\begin{align}
\frac{CM\varepsilon^\beta}{\alpha}R^2\leq    \frac{CM\varepsilon^{\beta-1}}{\alpha}R^3 \leq \frac{CL_{A} M\varepsilon^{\beta-1}}{\alpha}R^4.
\end{align}
We also assume $\alpha$ to be large. Then we are reduced to optimize the following term in (\ref{eq:E_explicit})
\begin{equation}
    			  C(\alpha-1)[\log\frac{r_3}{2r_2}+1]+
    			\frac{CM^2}{\alpha^2}\varepsilon^{2\beta}R^4
    			+\frac{CL_A M}{\alpha}\varepsilon^{\beta-1}R^4 .
    \end{equation}
  %  Fix $\alpha\ge 2$ and $R\in(0,r_0]$. Define
  %  \begin{equation}\label{eq:eps_choice_full}
   % 	\varepsilon:=\min\left\{\frac{R}{2},\ \left(\frac{\alpha}{KR}\right)^{\frac{1}{\beta+1}}\right\}.
  %  \end{equation}
    %Then
    We choose $\frac{M^2}{\alpha^2}\varepsilon^{2\beta}=
    			\frac{ M}{\alpha}\varepsilon^{\beta-1}$. Then
    $\varepsilon=(\frac{\alpha }{M})^{\frac{1}{\beta+1}}$.
    Hence \begin{equation}
    			  C(\alpha-1)[\log\frac{r_3}{2r_2}+1]+
    			{C}L_A^{\frac{2\beta}{\beta+1}} (\frac{M}{\alpha})^{\frac{2}{\beta+1}} R^4
    \end{equation}
  Next, to optimize the last inequality, we select $\alpha=(\frac{M}{\alpha})^{\frac{2}{\beta+1}} R^4 $, i.e.  $\alpha=M^{\frac{2}{\beta+3}} R^{\frac{4(\beta+1)}{\beta+3}}$.
  Thus, we obtain the best exponent $C_1 M^{\frac{2}{\beta+3}} R^{\frac{4(\beta+1)}{\beta+3}} (\log\frac{r_3}{2r_2}+1)$ in (\ref{eq:98_rederive_general}). Therefore, the three-ball inequality 
  \begin{equation}
				\|u\|_{L^2(B_{r_2})}\le \exp\!\left\{\,C_1 M^{\frac{2}{\beta+3}} R^{\frac{4(\beta+1)}{\beta+3}} (\log\frac{r_3}{2r_2}+1)\,\right\}\|u\|_{L^2(B_{r_1})}^{\theta}\|u\|_{L^2(B_{r_3})} ^{1-\theta}
					\end{equation}
   is arrived, where $0<\theta=\frac{\log \frac{r_3}{2r_2}}{e^{CL_A(r_3-r_1)}\log \frac{2r_2}{r_1} +\log \frac{r_3}{2r_2} } <1$.
  \end{proof}

At last, we provide a corollary for the quantitative three-ball inequality for 
$u$ satisfying
    \begin{align}
    -\operatorname{div}(A\nabla u)+W(x) \cdot \nabla u+V(x)u=0\ \text{in}\ B_R,
    \label{general-100}
    \end{align}
    where \begin{align}
    \label{WVWV-1}
    W(x)\in C^{0,\beta_0}, \ V(x)\in  C^{0,\beta} \quad \mbox{with} \  \ \|W\|_{C^{0,\beta_0}}\le K, \ \|V\|_{C^{0,\beta}}\le M.\end{align}
    Making use of the argument in Section 3 and current section, we can show that
\begin{corollary}\label{coro: Corollary on general case}
Let $u$ solve $(\ref{general-100})$
under the assumption (\ref{lips-1}) and (\ref{WVWV-1}).
Then there exists some universal constant $C(n, \lambda, \Lambda, L_A, \beta, \beta_0)$ such that for any $0<r_1<r_2<2r_2<r_3<\frac{R}{2}$,
    \begin{equation}\label{eq:general equation}
        \|u\|_{r_2}
    	\le
    	\exp\!\Bigg\{C\Bigl[M^{\frac{2}{\beta+3}} R^{\frac{4(\beta+1)}{\beta+3}} +K^{\frac{2}{\beta_0+1}}R^{\frac{2}{\beta_0+1}}\Bigl]
    	\Bigl[
    	1+2\log\!\Bigl(\frac{r_3}{2 r_2}\Bigr)
    	\Bigr]
    	\Bigg\}
    	\,
    	\|u\|_{r_1}^{\theta}\,
    	\|u\|_{r_3}^{1-\theta},
\end{equation}
 and 
\[
0<\theta=
\frac{\log r_3-\log 2r_2}
{\log r_3-\log 2r_2+Ce^{C(r_3-r_1)}\bigl[\log 2r_2-\log r_1\bigr]}<1.
\]
    \end{corollary}

\section{Appendix}
In this section, we present the proof of Lemma \ref{lem:hprime} and Lemma \ref{lem:Dprime_variable}, which are used in the section 4.
    	\begin{proof}[Proof of Lemma \ref{lem:hprime}]
    		
    		Note that $\nu$ in (\ref{eq:mu}) is independent of $r$.
           The differentiation of $H$ with respect to $r$ still acts on $e^{\alpha-1}_r$.  Hence
    		\begin{equation}\label{eq:Hprime_step1}
    			\begin{aligned}
    				H'(r)
    =\int_{B_r}|u|^2\nu\,\partial_r(e_r^{\alpha-1})\,dx 
    				=2(\alpha-1)r\int_{B_r}|u|^2\nu\,e_r^{\alpha-2}\,dx .
    			\end{aligned}
    		\end{equation}
    		By the fact that $r^2=e_r+|x|^2$, we have
    		\begin{equation}\label{eq:Hprime_minus}
    			H'(r)-\frac{2(\alpha-1)}{r}H(r)
    			=\frac{2(\alpha-1)}{r}\int_{B_r}|u|^2\nu\,|x|^2\,e_r^{\alpha-2}\,dx .
    		\end{equation}
    		The definition of $\nu$ in (\ref{eq:mu}) gives that
    $\nu(x)|x|^2=A(x)x\cdot x.$
    		Hence \eqref{eq:Hprime_minus} becomes
\begin{equation}\label{eq:Hprime_Ax}
    			H'(r)-\frac{2(\alpha-1)}{r}H(r)
    			=\frac{2(\alpha-1)}{r}\int_{B_r}|u|^2\,A(x)x\cdot x\,e_r^{\alpha-2}\,dx .
    		\end{equation}
    		%Since $\nabla e_r(x)=\nabla(r^2-|x|^2)=-2x$, we have
    		%\[
    		%\nabla(e_r^{\alpha-1})=(\alpha-1)e_r^{\alpha-2}\nabla e_r
    		%=-2(\alpha-1)x\,e_r^{\alpha-2}.
    		%\]
    		%Hence
    		%\[
    		%2(\alpha-1)x\,e_r^{\alpha-2}=-\nabla(e_r^{\alpha-1}).
    		%\]
    		Then \eqref{eq:Hprime_Ax} yields that
\begin{equation}\label{eq:Hprime_div_form}
    			H'(r)-\frac{2(\alpha-1)}{r}H(r)
    			=-\frac{1}{r}\int_{B_r}|u|^2\,A(x)x\cdot\nabla(e_r^{\alpha-1})\,dx .
    		\end{equation}
    		Integrating by parts gives 
            %and using that $e_r^{\alpha-1}=0$ on $\partial B_r$, we get
    		%\begin{equation}\label{eq:Hprime_ibp1}
    			%-\int_{B_r}|u|^2\,A x\cdot\nabla(e_r^{\alpha-1})
    			%=\int_{B_r}\operatorname{div} \big(|u|^2 A x\big)\,e_r^{\alpha-1}\,dx .
    		%\end{equation}
    		%Combining \eqref{eq:Hprime_div_form}--\eqref{eq:Hprime_ibp1}, 
\begin{equation}\label{eq:Hprime_expand_div}
    			H'(r)-\frac{2(\alpha-1)}{r}H(r)
    			=\frac{1}{r}\int_{B_r}2u\,A\nabla u\cdot x\,e_r^{\alpha-1}\,dx
    			+\frac{1}{r}\int_{B_r}|u|^2\,{\rm div} (Ax)\,e_r^{\alpha-1}\,dx .
    		\end{equation}

    		%Recall
    		%\begin{equation}\label{eq:I_def_410}
    		%	I_0(r):=2\alpha\int_{B_r}u\,A\nabla u\cdot x\,e_r^{\alpha-1}\,dx .
    		%\end{equation}
    		Thus, the first term on the right-hand side of \eqref{eq:Hprime_expand_div} equals $\frac{1}{\alpha r}I_0(r)$.
    		For the second term, we write ${\rm div}(Ax)=n\nu+\big({\rm div}(Ax)-n\nu\big)$ and obtain that
    		\begin{equation}\label{eq:split_divAx}
    			\frac{1}{r}\int_{B_r}|u|^2\,{\rm div}(Ax)\,e_r^{\alpha-1}\,dx
    			=\frac{n}{r}\int_{B_r}|u|^2\,\nu\,e_r^{\alpha-1}\,dx
    			+\frac{1}{r}\int_{B_r}|u|^2\big({\rm div}(Ax)-n\nu\big)e_r^{\alpha-1}\,dx .
    		\end{equation}
    		%Since $H(r)=\int_{B_r}|u|^2\nu e_r^{\alpha-1}$, 
            The first term on the right hand side of \eqref{eq:split_divAx}
    		is $\frac{n}{r}H(r)$. For the second term we define
    		\begin{equation}\label{eq:EH_def_again}
\mathcal{E}_H(r):=\int_{B_r}|u|^2\Big[{\rm div}(Ax)\,\nu^{-1}-n\Big]\nu\,e_r^{\alpha-1}\,dx ,
    		\end{equation}
    	 The last three inequalities give that
%\begin{equation}\label{eq:Hprime_pre413}
    		%	H'(r)-\frac{2(\alpha-1)}{r}H(r)=\frac{n}{r}H(r)+\frac{1}{\alpha r}I_0(r)+\frac{1}{r}E_H(r),
    		%\end{equation}
    		%hence
\begin{equation}\label{eq:Hprime_413_again}
    			H'(r)=\frac{2(\alpha-1)+n}{r}H(r)+\frac{1}{\alpha r}I_0(r)+\frac{1}{r}\mathcal{E}_H(r).
    		\end{equation}
    		Recall \eqref{eq:I0IE}, we obtain the derivative of $H$ in (\ref{eq:Hprime_final_413}). From (\ref{AAA-1}), we get
    		\[
    		|\mathcal{E}_H(r)|
    \leq cL_Ar\int_{B_r}|u|^2\nu\,e_r^{\alpha-1}\,dx
    		=c L_A\,r\,H(r),
    		\]
    		which is \eqref{eq:EH_bound_414}.
    		This completes the proof.
    	\end{proof}

   The proof of Lemma \ref{lem:Dprime_variable} is derived from integration by parts argument as well.
     	\begin{proof}[Proof of Lemma \ref{lem:Dprime_variable}]
    		Recall $D(r)$ in (\ref{eq:DLEdef}).
    		Differentiating with respect to $r$. gives 
\begin{equation}\label{eq:Dprime_step1}
    			\begin{aligned}
    				D'(r)
    				=2\alpha r\int_{B_r} a_{ij}\partial_i u\,\partial_j u\,e_r^{\alpha-1}\,dx.
    			\end{aligned}
    		\end{equation}
    		Splitting $r^2=e_r+|x|^2$ in the last integral, we obtain
\begin{equation}\label{eq:Dprime_step2}
    			\begin{aligned}
    				D'(r)
    				&=\frac{2\alpha}{r}\int_{B_r} a_{ij}\partial_i u\,\partial_j u\,e_r^{\alpha}\,dx
    				+\frac{2\alpha}{r}\int_{B_r} a_{ij}\partial_i u\,\partial_j u\,|x|^2\,e_r^{\alpha-1}\,dx \\
    				&=\frac{2\alpha}{r}D(r)+\frac{2\alpha}{r}\int_{B_r} a_{ij}\partial_i u\,\partial_j u\,|x|^2\,e_r^{\alpha-1}\,dx.
    			\end{aligned}
    		\end{equation}

    	Recall $\omega(x)$ is given in (\ref{eq:Zdef}).  Then $\omega(x)\cdot x=|x|^2$. 
    		It holds that
\begin{equation}\label{eq:divZomega}
    			{\rm div} (\omega e_r^\alpha)
    			=({\rm div} \omega)\,e_r^\alpha+\omega\cdot\nabla(e_r^\alpha)
    			=({\rm div} \omega)\,e_r^\alpha-2\alpha|x|^2e_r^{\alpha-1}.
    		\end{equation}
    		Set $F(x):=a_{ij}(x)\partial_i u(x)\partial_j u(x)$. Multiplying \eqref{eq:divZomega} by $F$ and integrating over $B_r$ gives
    		\begin{align}\label{eq:rewrite_last_integral}
    			2\alpha\int_{B_r}F|x|^2e_r^{\alpha-1}\,dx
    			&=\int_{B_r}F({\rm div}\omega)\,e_r^\alpha\,dx-\int_{B_r}F\,{\rm div}(\omega e_r^\alpha)\,dx \notag\\
    			&=\int_{B_r}F({\rm div}\omega)\,e_r^\alpha\,dx+\int_{B_r}\nabla F\cdot \omega\,e_r^\alpha\,dx.
    		\end{align}
    		%In the last step we integrated by parts:
    		%\[
    		%\int_{B_r}F\,{\rm div}(\omegae_r^\alpha)\,dx
    		%=\int_{\partial B_r}F\,\omegae_r^\alpha\cdot\nu\,dS-\int_{B_r}\nabla F\cdot \omega\,e_r^\alpha\,dx
    		%=-\int_{B_r}\nabla F\cdot \omega\,e_r^\alpha\,dx,
    		%\]
    		%because $e_r^\alpha=0$ on $\partial B_r$. 
            %Writing $\nabla F\cdot \omega=\partial_kF\,\omega_k$ and decomposing
    		Let ${\rm div} \omega=n+(\rm div \omega-n)$. It follow from \eqref{eq:Dprime_step2} and \eqref{eq:rewrite_last_integral} that
\begin{equation}\label{eq:Dprime_step3}
    			D'(r)=\frac{2\alpha+n}{r}D(r)+\frac{1}{r}\int_{B_r}\partial_kF\,\omega_k\,e_r^\alpha\,dx
    			+\frac{1}{r}\int_{B_r}({\rm div} \omega-n)\,F\,e_r^\alpha\,dx.
    		\end{equation}
    		
    		We expand
    		\[
    		\partial_kF=\partial_k(a_{ij}\partial_i u\,\partial_j u)
    		=(\partial_k a_{ij})\partial_i u\,\partial_j u
    		+a_{ij}(\partial_{ik}u)\,\partial_j u
    		+a_{ij}\partial_i u\,(\partial_{jk}u).
    		\]
    		Using $a_{ij}=a_{ji}$, the last two terms are equal after swapping $(i,j)$. Hence
    		\begin{equation}\label{eq:T0}
    			\int_{B_r}\partial_kF\,\omega_k\,e_r^\alpha\,dx
    			=\int_{B_r}(\partial_k a_{ij})\partial_i u\,\partial_j u\,\omega_k\,e_r^\alpha\,dx
    			+2\int_{B_r}a_{ij}(\partial_{ik}u)\,\partial_j u\,\omega_k\,e_r^\alpha\,dx.
    		\end{equation}
    		For the second integral in \eqref{eq:T0}, integrating by parts in the $i$th variable yields
    		\begin{align}\label{eq:IBP_second_derivatives}
    			2\int_{B_r}a_{ij}(\partial_{ik}u)\,\partial_j u\,\omega_k\,e_r^\alpha\,dx
    			&=-2\int_{B_r}\partial_k u\,\partial_i\!\big(a_{ij}\partial_j u\,\omega_k\,e_r^\alpha\big)\,dx \notag\\
    			&=-2\int_{B_r}\partial_k u\,\omega_k\,\partial_i(a_{ij}\partial_j u)\,e_r^\alpha\,dx 
    	-2\int_{B_r}a_{ij}\partial_k u\,\partial_j u\,(\partial_i \omega_k)\,e_r^\alpha\,dx \notag\\
    			&\quad +4\alpha\int_{B_r}a_{ij}\partial_j u\,x_i\;(\partial_k u\,\omega_k)\,e_r^{\alpha-1}\,dx. 
    		\end{align}
    		%again with no boundary term since $e_r^\alpha=0$ on $\partial B_r$.
    		
    		%We rewrite the last term in \eqref{eq:IBP_second_derivatives} using
    		%$\partial_i(e_r^\alpha)=-2\alpha x_ie_r^{\alpha-1}$:
    		%\begin{align}\label{eq:last_term_square}
    		%	-2\int_{B_r}a_{ij}\partial_k u\,\partial_j u\,\omega_k\,\partial_i(e_r^\alpha)\,dx
    		%	&=\notag\\
    			%&=4\alpha\int_{B_r}(A\nabla u\cdot %x)\,(\nabla u\cdot \omega)\,e_r^{\alpha-1}\,dx.
    		%\end{align}

    		Since $\omega=\nu^{-1}Ax$, we have $\nabla u\cdot \omega=\nu^{-1}(A\nabla u\cdot x)$. Then
	\begin{equation}\label{eq:square_identity}
    			4\alpha\int_{B_r}a_{ij}\partial_j u\,x_i\;(\partial_k u\,\omega_k)\,e_r^{\alpha-1}\,dx 
    			=4\alpha\int_{B_r}(A\nabla u\cdot x)^2\,\nu^{-1}\,e_r^{\alpha-1}\,dx.
    		\end{equation}
    		Next, we decompose the second  term in the right hand side of  \eqref{eq:IBP_second_derivatives} by adding and subtracting $\delta_{ik}$,
\begin{align}\label{eq:middle_term_decomp}
    			-2\int_{B_r}a_{ij}\partial_k u\,\partial_j u\,(\partial_i \omega_k)\,e_r^\alpha\,dx
    			&=2\int_{B_r}a_{ij}\partial_k u\,\partial_j u\,(\delta_{ik}-\partial_i \omega_k)\,e_r^\alpha\,dx
    			-2\int_{B_r}a_{ij}\partial_k u\,\partial_j u\,\delta_{ik}\,e_r^\alpha\,dx \notag\\
&=2\int_{B_r}a_{ij}\partial_k u\,\partial_j u\,(\delta_{ik}-\partial_i \omega_k)\,e_r^\alpha\,dx
    			-2D(r).
    		\end{align}
    		Finally, the first term in the right hand side of \eqref{eq:IBP_second_derivatives} becomes
    		\begin{equation}\label{eq:div_term}
    			-2\int_{B_r}\partial_k u\,\omega_k\,\partial_i(a_{ij}\partial_j u)\,e_r^\alpha\,dx
    			=-2\int_{B_r}(A\nabla u\cdot x)\,\nu^{-1}\,{\rm div}(A\nabla u)\,e_r^\alpha\,dx.
    		\end{equation}
                		Combining \eqref{eq:T0}--\eqref{eq:div_term}, we obtain the identity
    \begin{equation}\label{eq:Tfinal}
    			\begin{aligned}
    				\int_{B_r}\partial_kF\,\omega_k\,e_r^\alpha\,dx
    				=&\ -2D(r)
    				+4\alpha\int_{B_r}(A\nabla u\cdot x)^2\,\nu^{-1}\,e_r^{\alpha-1}\,dx
    				-2\int_{B_r}(A\nabla u\cdot x)\,\nu^{-1}\,{\rm div}(A\nabla u)\,e_r^\alpha\,dx \\
    				&\ +2\int_{B_r}a_{ij}\partial_k u\,\partial_j u\,(\delta_{ik}-\partial_i \omega_k)\,e_r^\alpha\,dx
    				+\int_{B_r}(\partial_k a_{ij})\partial_i u\,\partial_j u\,\omega_k\,e_r^\alpha\,dx.
    			\end{aligned}
    		\end{equation}
    		
    		Substituting \eqref{eq:Tfinal} into \eqref{eq:Dprime_step3} gives	\begin{equation}\label{eq:Dprime_step4}
    			\begin{aligned}
    				D'(r)=&\ \frac{2\alpha+n-2}{r}D(r)
    				+\frac{4\alpha}{r}\int_{B_r}(A\nabla u\cdot x)^2\,\nu^{-1}\,e_r^{\alpha-1}\,dx
    				\\
    				&-\frac{2}{r}\int_{B_r}(A\nabla u\cdot x)\,\nu^{-1}\,{\rm div}(A\nabla u)\,e_r^\alpha\,dx+\frac{1}{r}\mathcal{E}_D(r),
    			\end{aligned}
    		\end{equation}
    		where the error term is given as
    		\begin{equation}
    				\begin{aligned}
    					\mathcal{E}_D(r):=&\ \int_{B_r}({\rm div} \omega-n)\,a_{ij}\partial_i u\,\partial_j u\,e_r^\alpha\,dx
    					+2\int_{B_r}a_{ij}\partial_k u\,\partial_j u\,(\delta_{ik}-\partial_i \omega_k)\,e_r^\alpha\,dx \\
    					&\ +\int_{B_r}(\partial_k a_{ij})\partial_i u\,\partial_j u\,\omega_k\,e_r^\alpha\,dx.
    			\end{aligned}
    		\end{equation}
    	
    	The bound \eqref{eq:ED_bound} follows from \eqref{eq:Zbounds} that
    	\[
    	|\mathcal{E}_D(r)|\le C\int_{B_r}L_A|x|\ |\nabla u|^2\,e_r^\alpha\,dx \le  C L_A r\,D(r).
    	\]
        This finishes the proof of  Lemma \ref{lem:Dprime_variable}.
    	\end{proof}

        \textbf{Data Availability Statements:} All data generated or analyzed during this study are included in this published
article.
    
    \bibliography{reference}
\bibliographystyle{abbrv}

\end{document}